\documentclass{amsart}

\RequirePackage{amsthm,amsmath,amsfonts,amssymb}
\RequirePackage[numbers]{natbib}
\RequirePackage[colorlinks,citecolor=blue,urlcolor=blue]{hyperref}
\RequirePackage{graphicx}

\usepackage{amssymb}
\usepackage{bbm}
\usepackage{dsfont}
\usepackage{enumitem}
\usepackage{ifpdf}
\usepackage{mathtools}
\usepackage{psfrag}
\ifpdf\usepackage{pst-pdf}\else\fi

\theoremstyle{plain}
\newtheorem{theorem}{Theorem}[section]
\newtheorem{lemma}[theorem]{Lemma}
\newtheorem{corollary}[theorem]{Corollary}
\newtheorem{example}[theorem]{Example}
\theoremstyle{remark}
\newtheorem{definition}[theorem]{Definition}
\newcommand{\nmb}[2]{\ifx!#1{\ref{nmb:#2}}%
\else\if.#1{\label{nmb:#2}}%
\else\if0#1{\label{nmb:#2}}%
\else{{#2}}%
\fi\fi\fi}
\newcommand{\Cut}{\mathrm{Cut}}
\newcommand{\iid}{i.i.d.\@}
\def\o{\circ}

\def\al{\alpha}

\def\si{\sigma}
\def\ta{\tau}
\def\ph{\varphi}

\def\ps{\psi}

\def\Om{\Omega}
\def\d{\operatorname{d}\!}
\def\i{^{-1}}
\def\x{\times}

\def\Id{\operatorname{Id}}

\def\Aut{\operatorname{Aut}}

\def\S{S}

\def\AMSonly#1{}
\def\Cov{\operatorname{Cov}}

\def\proj{\pi}
\newcommand{\partition}[1]{(\mathbf #1)}

\def\probspace{\mathcal P}

\title[]{Geometry of Sample Spaces}

\author[P.~Harms]{Philipp Harms}
\address{Nanyang Technological University, Singapore}
\email{philipp.harms@ntu.edu.sg}

\author[P.~Michor]{Peter W. Michor}
\address{Faculty of Mathematics, University of Vienna, Austria}
\email{peter.michor@univie.ac.at}

\author[X.~Pennec]{Xavier Pennec}
\address{Universtit\'e C\^ote d'Azur and Inria, Sophia Antipolis, France}
\email{xavier.pennec@inria.fr}

\author[S.~Sommer]{Stefan Sommer}
\address{Department of Computer Science, University of Copenhagen, Denmark}
\email{sommer@di.ku.dk}



\begin{document}

\begin{abstract}
In statistics, independent, identically distributed random samples do not carry a natural ordering, and their statistics are typically invariant with respect to permutations of their order. Thus, an $n$-sample in a space $M$ can be considered as an element of the quotient space of $M^n$ modulo the permutation group. The present paper 
takes this definition of sample space and the related concept of orbit types as a starting point for developing a geometric perspective on statistics. We aim at deriving a general mathematical setting for studying the behavior of empirical and population means in spaces ranging from smooth Riemannian manifolds to general stratified spaces.
 
We fully describe the orbifold and path-metric structure of the sample space when $M$ is a manifold or path-metric space, respectively. These results are non-trivial even when $M$ is Euclidean. We show that the infinite sample space exists in a Gromov--Hausdorff type sense and coincides with the Wasserstein space of probability distributions on $M$. We exhibit Fr\'echet means and $k$-means as metric projections onto 1-skeleta or $k$-skeleta in Wasserstein space, and we define a  new and more general notion of polymeans. This geometric characterization via metric projections applies equally to sample and population means, and we use it to establish asymptotic properties of polymeans such as consistency and asymptotic normality.
\end{abstract}

\maketitle

\section{Introduction} \nmb0{1}

Following the pioneering developments of directional statistics \cite{jupp89} 
and shape statistics \cite{kendall_shape_1984,kendall_shape_1999,dryden_statistical_1998}, there is a growing need in  many application domains for the statistical analysis of populations of objects in complicated non-Euclidean spaces. One can cite for instance tree-spaces in biology \cite{BHV01}, Riemannian manifolds and Lie groups, including diffeomorphism groups, in medical image analysis and computer vision \cite{pennec_intrinsic_2006,pennec_sommer_fletcher_2020,turaga_riemannian_2016}, or more generally stratified spaces \cite{marron_overview_2014}. 
With the choice of a relevant distance, a natural generalization of the central values of a population of objects in these spaces is the Fr\'echet $p$-mean, that is the set of  minima of the mean distance to the power $p$~\cite{frechet_les_1948}. While the choice of $p=2$ is often used because it corresponds to the usual arithmetic, lower values of $p$ up to $p=1$ defining the median (``valeur equiprobable'' in Fr\'echet's words) are also often useful for robust statistics. 

This paper develops a general mathematical setting to study the behavior of empirical and population Fr\'echet $p$-means in spaces ranging from smooth Riemannian manifolds to general stratified spaces. We start from the key observation that independent, identically distributed (\iid{}) random samples do not carry a natural ordering, and their statistics are typically invariant with respect to permutations of their order. 
Thus, an $n$-sample in a space $M$ can naturally be considered as an element of the quotient space $M^n/\S_n$ of $n$-tuples modulo the permutation group $\S_n$. 
This space shall accordingly be called \emph{sample space}. 
The paper takes this definition as a starting point for developing a geometric perspective on statistics, guided by the notion of orbit type. 
This way, we provide a theoretical basis for further investigations on unordered samples in non-Euclidean spaces.

\subsection{Background}
For non-positively curved spaces in the sense of Alexandrov, the 2-mean is always unique when it exists \cite{sturm_probability_2003}. For positively curved Riemannian manifolds, an important effort has been spent in determining the convexity conditions on the distribution that ensure uniqueness \cite{karcher_riemannian_1977,buser_gromovs_1981,afsari_riemannian_2011}. 
However, many very useful distributions such as wrapped or truncated Gaussian distributions on the tangent spaces do not fulfill these conditions even if they have a unique Fr\'echet mean. 

Asymptotic properties of the sample mean for distributions on Riemannian manifolds with a unique population Fr\'echet mean were  studied by Bhattacharya and Patrangenaru \cite{bhattacharya_nonparametric_2002,bhattacharya_large_2003,bhattacharya_large_2005}. In particular, they showed the consistency of the sample Fr\'echet mean $\bar x_n$ of $n$ \iid{} samples of a random variable $x$ for large sample sizes (law of large numbers), building on a strong consistency result of \cite{ziezold_expected_1977}. Under the Karcher and Kendall convexity conditions for the uniqueness of the population mean $\bar x$, the Bhattacharya-Patrangenaru central limit theorem (CLT) further states that the random variables $u_n = \sqrt{n}  \log_{\bar x} (\bar x_n)$ converge in distribution to the Gaussian $\mathcal{N}(0, {\bar H}^{-1} \text{Cov}(x) {\bar H}^{-1})$ in the tangent space at $\bar x$ whenever the expected Hessian $\bar H$ of half the Riemannian squared distance at the population mean $\bar x$ is invertible. 
This type of CLT based on the delta method was further generalised in \cite{kendall_limit_2011} to non-\iid{} variables and in \cite{huckemann_intrinsic_2011} to  summary statistics other than the mean, such as principal geodesics. 

In non-manifold stratified spaces of negative curvature, an intriguing phenomenon  was discovered 10 years ago: the Fr\'echet mean may be sticky on singular strata \cite{hotz_sticky_2013}. 
A regular random variable (that is a not fully concentrated on singular strata) whose Fr\'echet mean is located on a singular stratum is said to have a sticky mean if a sufficiently small variation of that  random variable continues to have its Fr\'echet mean on the singular stratum. In other words, the singular strata are attractive. It is surprising that a CLT can still be derived under these conditions \cite{hotz_sticky_2013}. This suggests that some regularity can be used for deriving CLTs in more general settings. 
Stickiness does not seem to happen in positive curvature. For instance, Kendall shape spaces in three or higher dimensions are stratified, but the Fr\'echet mean of regular random variables was shown to belong to the top regular stratum (manifold-stability) 
 \cite{huckemann_meaning_2012}. 
In other words, singular strata of that kind are repulsive. 

More recently, an apparently opposite unusual behavior of the CLT was discovered with smeary means, where the empirical Fr\'echet means converge at an asymptotic rate lower than $\sqrt{n}$; see \cite{eltzner_smeary_2019} e.g. Other results show that intermediate repulsive or attractive behaviours  can happen on Riemannian manifolds, controlled either by the curvature \cite{pennec_curvature_2019,eltzner_geometrical_2019} or by the topology \cite{hundrieser_finite_2020}. Thus,  classical tests based on asymptotic results for Euclidean spaces might be biased, which is a critical problem for many applications. This highlights the need for a new mathematical framework to study the distribution of the empirical Fr\'echet mean,
either in the small sample regime or asymptotically.

While considering $n$-samples disregarding ordering is not new, the literature is sparse in linking geometric properties of the quotient space to sample statistics. In the Euclidean case, de Finetti's theorem \cite{finetti_funzione_1930,finetti_prevision_1937} and the theory of Hewitt and Savage \cite{hewitt_symmetric_1955} on exchangeability and presentability characterized distributions invariant to finite permutations leading to central limit theorems based on exchangeability instead of independence \cite{chernoff_central_1958,blum_central_1958,klass_central_1987}. 
We here develop a similar theory using additional geometric structures.

\subsection{Overview and results}

The convenient level of generality that we adopt is that of path-metric spaces \cite{gromov1999metric, burago2001course}, see \nmb!{8.1}, where the distance is given by the infimum of the length of curves joining the two points; for complete path-metric spaces the infimum is a minimum, see \nmb!{8.2}.

We first describe in Section~\nmb!{2} the orbifold (resp. path-metric) structure of the sample space $M^n/\S_n$ when $M$ is a manifold (resp. a path-metric space). These results are non-trivial even when $M$ is Euclidean but well known in the realm of reflection groups and Weyl chambers. The sample space $M^n/\S_n$ can be stratified by the number of pairwise distinct points. The regular part $(M^n/\S_n)_{\text{reg}}$ contains the unordered configurations where the $n$ points are distinct. The lower dimensional strata are called the $q$-skeleta, see \nmb!{2.2}, and comprise unordered configurations with exactly $q < n$ distinct points. 
A finer stratification classifying orbit types is based on the partition $\partition{k}\coloneqq(k_1\ge\dots\ge k_q)$ of $n$  describing the number of identical points; see \nmb!{2.5} and \nmb!{2.6}. 
Sub-partitioning (distinguishing some of the points that were previously identified) gives a half-ordering on partitions which are thus organized in a geometric lattice structure. 
The orbit-type stratum $(M^n/\S_n)_{(n)}$ with the smallest partition $(n)$ is the \emph{diagonal} $\{x:x_1=\dots=x_n\} \simeq M$ where all points coincide. This is the 1-skeleton, which can be identified with the base manifold $M$. At the other end of the lattice, the regular orbit stratum  $(M^n/S_n)_{\text{reg}} = (M^n/S_n)_{(1\ge1\ge\dots\ge1)}$ is the open, dense, connected, and locally connected subset of all unordered configurations with $n$ distinct points. The closure of $(M^n/\S_n)_{\partition{k}}$ in $M^n/\S_n$ is the disjoint union of all $(M^n/\S_n)_{\partition{k'}}$ with $\partition{k'}\le \partition{k}$; see \nmb!{2.10}. The $q$-skeleton of $M^n/\S_n$ is the the union of all orbit strata $(M^n/\S_n)_{\partition{k}}$ corresponding to all partitions 
$\partition{k}=(k_1\ge\dots \ge k_p)$  with  length $p\le q \leq n$. 
The projection to $q$-skeleta and orbit strata will be used in Section \nmb!{5} to characterize the Fr\'echet $p$-mean and to define a generalization called polymeans.   

Section~\nmb!{3} investigates the metric properties of the sample spaces when we assume that  $M$ is a complete path-metric space. 
The  $L_p$ metric $d_p(x,y) = \left(\frac1n\sum_{i=1}^n d(x_i,y_i)^p\right)^{1/p}$ with $p \in [1,\infty)$ on $M^n$ induces a canonical quotient metric on the sample space $(M^n/\S_n, \bar d_p)$, which is then a complete path-metric space; see \nmb!{3.2}. Moreover, orbit-type strata have convex closures, and a minimizing geodesic in the sample space $(M^n/\S_n,\bar d_p)$ is the projection of a minimizing geodesic in the configuration space $(M^n,d_p)$. When $M$ is Riemannian and $p=2$, one can show that geodesics are more regular at interior points than at their end-points, \nmb!{3.7}. However, this assertion is generally wrong for non-Riemannian complete path-metric spaces, like for instance for the 3-spider, \nmb!{3.8}. This lack of regularity could be linked to stickiness. 

In order to investigate sub-samples (bootstrap) and infinite samples together in the same space, we show in \nmb!{4.7} that the sample space $(M^n/\S_n,\bar d_p)$ is isometric to the space of mixtures of $n$-atomic measures (the empirical law of the samples) endowed with the $p$-Wasserstein metric. Moreover, the infinite sample space $\lim_{n\to\infty}M^n/\S_n$ exists in a weakened Gromov--Hausdorff type sense and coincides with the $p$-Wasserstein space $(\probspace^p(M),\bar d_p)$ of $p$-integrable probability distributions on $M$; see \nmb!{4.8}. The extension of skeleta and orbit-type strata to infinite sample spaces can then be done easily: the \emph{$q$-skeleton} in the infinite-sample space $\probspace^p(M)$ is the subset $\probspace(M)_q$ of all probability distributions with at most $q$ support points; see \nmb!{4.11}. 
Similarly, for any partition $\partition{k} \coloneqq (w_1\geq\cdots\geq w_q)$ consisting of non-negative weights $w_i$ summing up to $1$, the \emph{$\partition{k}$-stratum} in the infinite-sample space $\probspace^p(M)$ is the subset of mixtures $P=\sum_{i=1}^q w_i \delta_{x_i} \in \probspace(M)_q$ with $q$ distinct points $x_i$. It is interesting to note that such a mixture of $q$ Diracs is realized in a finite sample space for some $n$ if the weights are all rational, but irrational weights can only be achieved in the infinite-sample limit.

With this setting, we are in position to exhibit in Section \nmb!{5}  empirical and population Fr\'echet means as metric projections onto the 1-skeletum in sample space or Wasserstein space, and we define a  new and more general notion of empirical and population  polymeans by the projection on the $q$-skeleta $(M^n/\S_n)_q$  or on the $\partition{k}$-strata $(M^n/\S_n)_{\partition{k}}$. These polymeans can be interpreted as the clusters of the well known $k$-means clustering algorithm: the $k$ distinct points are the cluster centroids (we also call them the unweighted polymeans) and the weights $w_i$ are the relative masses of the clusters. 
As everything is defined for $p$-integrable distributions ($p \geq 1$), our definitions are actually valid for general Fr\'echet $p$-means and $p$-power $k$-means. Since $q$-skeleta and $\partition{k}$-strata are closed in all sample spaces, as well as in the $p$-Wasserstein space, the existence of empirical and population polymeans is ensured. The uniqueness is a much harder problem. In the Riemannian case with $p=2$, recent results on the regularity of the singular set of the distance to a sufficiently regular set show that empirical polymeans of \iid{} samples with an absolutely continuous law are almost surely unique. This partly extends the previous result of \cite{arnaudon_means_2014} on the uniqueness of the empirical Fr\'echet $p$-mean.  

We turn in Section \nmb!{6} to probability distributions on sample spaces. It turns out that the correct space of infinite samples is not the quotient space $M^{\mathbb N}/\S_{(\mathbb N)}$ but the space $\probspace(M)$ of probability distributions on $M$. Indeed, using this definition one obtains as in the theory of Hewitt and Savage \cite{hewitt_symmetric_1955} that probability distributions on infinite sample spaces correspond exactly to symmetric probability distributions on configuration spaces, which in turn correspond exactly to mixtures of product distributions. This definition is also in line with the infinite-sample limit \nmb!{4.8}. The analogous statement for random variables instead of probability distributions is that random samples correspond exactly (possibly after passing to an extended probability space) to conditionally \iid{} random configurations; see \nmb!{6.6}. 

This setting allows us to establish in Section \nmb!{7} asymptotic properties of polymeans. 
We first show that the empirical $q$-means are strongly consistent estimators for the population $q$-means, in the sense that any accumulation point of the sets of empirical $q$-means is a population $q$-mean. Thus, when the population $q$-mean is unique, any measurable selection of empirical $q$-means converges in probability to the population $q$-mean, and we may inquire about the rate of convergence. We derive in \nmb!{7.4} an upper bound on the convergence rate of empirical $q$-means to the population $q$-mean. The bound depends first on the convergence rate in Wasserstein space of empirical distributions---a well studied subject---and second on the subspace geometry of the $q$-skeleton within Wasserstein space---a purely geometric question.
It remains an open problem if the bound is sharp and if $q$-means are asymptotically normal after a suitable normalization. However, when $M$ is a Riemannian manifold, we establish in \nmb!{7.6} the asymptotic normality of unweighted $q$-means for any $p\geq 1$ under mild conditions (null measure of the union of the cut loci of the centroids and of their ``medial axis'' and non-degenerate expected Hessian of the power $p$ distance to the closest centroid). We further refine this central limit theorem in \nmb!{7.7} from \iid{} to exchangeable sequences under some additional conditional independence assumptions. 

In the appendix we collect some tools from path-metric geometry.

\subsection{Open problems and future work}

Our framework opens the door to many further investigations by linking two traditionally distinct strands of literature, namely, statistics on manifolds and orbifold or path-metric geometry.
Tools from these fields can be fruitfully combined.
The setup is fully general and applies to curved spaces and more general stratified spaces, as needed in the previously cited applications. It also encompasses Fr\'echet $p$-means and not only the classical 2-mean, which opens the way to many useful asymptotic results for robust statistics. 

Our results also suggest that the non-standard convergence rates in the CLT are not only due to the geometry of $M$ but also the subspace geometry of the $k$-skeleta within the sample spaces.
For instance, considering the Fr\'echet mean as a projection on the $1$-skeleton casts a new geometric light on the uniqueness problem: in a Riemannian  manifold, it is unique whenever there is no mass on the singular set of the distance function to the 1-skeleton. Thus, one can conjecture that the geometry of the ``medial axis'' of the $q$-skeleton in $p$-Wasserstein space controls the uniqueness of the polymeans and that advances on the sub-space geometry of this set within Wasserstein space would extend this uniqueness theorem to more general settings. 

Likewise, the rate of convergence of the empirical 2-mean towards the population 2-mean is controlled by the eigenvalues of the expected Hessian of the squared distance (Corollary~\nmb!{7.7}).
The convergence rate towards the limiting distribution 
in the direction of an eigenvector falls below $\sqrt{n}$  whenever the corresponding eigenvalue vanishes. 
Conversely, stickiness could be induced by eigenvalues going to infinity. This last behavior cannot happen in smooth Riemannian manifolds, but it can be approached by concentrating the curvature at singular points. This could be a way to study stickiness on smoothable manifolds. 
For \iid{} samples with distribution $P$, we conjecture that these condition could be linked to the convexity or concavity of the geodesic distance in Wasserstein space from $P$ to the polymean in the $k$-skeleton, and thus that it can be controlled by some kind of Ma--Trudinger--Wang (MTW) condition \cite{figalli_convexity_2015}. 

\section{Orbit type stratification of sample spaces}\nmb0{2} 

Let $M$ be a topological space. 
For any natural number $n \in \mathbb N_{>0}$, the permutation group $\S_n$ of $n$ symbols acts on the $n$-fold product $M^n$ by permutation of the components. 
In symbols, we shall write $x_\sigma\coloneqq x\circ \sigma$ for the action of $\sigma \in \S_n$ on $x \in M^n$. 

\begin{definition}[Configurations and samples]
\nmb.{2.1}
An \emph{$n$-point configuration} or \emph{ordered $n$-sample} is an element of $M^n$, and this space is called \emph{(ordered) configuration space}.
An \emph{$n$-sample} is an element of the quotient space $M^n/\S_n$, and this space is called \emph{sample space} or \emph{unordered configuration space}.
The \emph{projection} is denoted by $\proj\colon M^n\to M^n/\S_n$.
\end{definition}

Note that this definition of configuration spaces differs from the one commonly used in topology, where the points are required to be pairwise distinct.
The set of pairwise distinct points is an open subset of $M^n$, and its fundamental group in the case $M=\mathbb R^2$ is the \emph{braid group}.
In contrast, we also consider the case where only $q<n$ points are mutually distinct:

\begin{definition}[Skeleta]
\nmb.{2.2}
A configuration $(x_1,\dots,x_n)$ is said to belong to the \emph{$q$-skeleton} if it consists of at most $q\in\mathbb N$ distinct points $x_i$. 
As the number of distinct points is $\S_n$-invariant, there is a corresponding notion of $q$-skeleta of samples.
\end{definition}

The name skeleton is taken from the theory of simplicial complexes and cell complexes. 
The filtration of sample space into skeleta is rather coarse, and finer stratifications are needed to fully describe the local geometry of sample space.
This is done next.

\begin{definition}[Orbifolds {\cite{wiki:orbifolds}}]
\nmb.{2.3}
A Hausdorff topological space $\mathcal O$ is an orbifold, if the following data are given: 
\begin{itemize}
\item 
An open cover $(U_i)$ of $\mathcal O$ which is closed under forming finite intersections. 
\item 
For each $i$ there is an
open subset 
$V_i\subset \mathbb R^N$ which is invariant under a faithful linear action of a finite group $G_i$ on $\mathbb R^N$ and a $G_i$-invariant quotient map $\proj_i\colon V_i\to V_i/G_i \cong U_i$.
\item 
If $U_i\subset U_j$ then there is an injective group homomorphism $\ph_{ij}:G_i\to G_j$  and a gluing map $\ps_{ij}$ from $V_i$ to an open subset of $V_j$ which is $G_i$-equivariant in the sense that $\ps_{i,j}(g.x)= \ph_{ij}(g).\ps_{ij}(x)$ for all $x\in V_i$ and such that $\proj_j\o \ps_{ij} = \proj_i$.
\end{itemize}
In this situation $(V_i,\proj_i, G_i)$ is then called an orbifold chart. 
\end{definition}

\begin{lemma}[Orbifold structure of sample space]
\nmb.{2.4}
If $M$ is a manifold, then the sample space $M^n/\S_n$ is an orbifold. 
\end{lemma}

\begin{proof}
For any $x\in M^n$, choose a chart $(U_i, u_i\colon U_i \to \mathbb R^m)$ such that whenever $x_i=x_j$ we have $(U_i,u_i)=(U_j,u_j)$. 
Then $u_1(U_1)\x \cdots\x u_n(U_n) \subseteq (\mathbb R^m)^n$ is invariant under the isotropy group $(S_n)_x$ and $\proj\o (u_1\i\x\cdots \x u_n\i)\colon u_1(U_1)\x \cdots\x u_n(U_n) \to \proj(U_1\x\cdots\x U_n)\subset M^n/S_n$ is the required orbifold chart. 
\end{proof}

The proof of \nmb!{2.4} shows more generally that the quotient space of a smooth manifold with respect to a properly discontinuous action of a group is an orbifold; in this case it is sometimes called a \emph{developable} or (by Thurston) a \emph{good} orbifold. 
To understand the orbifold structure of sample space, one has to describe the different \emph{orbit types}.

\begin{definition}[Orbit types]
\nmb.{2.5}
The orbit type of an ordered sample $x\in M^n$ is defined as the conjugacy class of its isotropy group $(\S_n)_x\coloneqq\{\si\in \S_n: x_\sigma=x\}$. 
As the orbit type is $S_n$-invariant, there is a corresponding notion of orbit types of samples in $M^n/\S_n$.
\end{definition}

The following theorem classifies the orbit types of sample space. 
It turns out that there are many different orbit types, one for each partition of the integer $n$.
This highlights the complicated geometry of sample space.

\begin{theorem}[Classification of orbit types]
\nmb.{2.6}
The orbit types in the configuration space $M^n$ are exactly given by the integer partitions of $n$ of the form
$$
n= k_1 + k_2 + \dots + k_q,\quad k_1\ge k_2\ge \dots \ge k_q\ge 1.
$$
We write $\partition{k}\coloneqq(k_1\ge\dots\ge k_q)$ for such a partition.
\end{theorem}

\begin{proof}
This follows from the fact that a point $x \in M^n$ is fixed by a permutation
$$
\si = (\si_1 \si_2 \dots \si_{k_1})(\si_{k_1+1}\dots\si_{k_1+k_2})\dots(\si_{k_1+\dots k_{q-1}+1}\dots \si_{k_1+\dots+k_p}) \in \S_n
$$ 
if and only if 
\begin{multline*}
x_{\si_1}=x_{\si_2}=\dots=x_{\si_{k_1}}, 
x_{\si_{k_1+1}}=\dots=x_{\si_{k_1+k_2}}, \quad\dots 
\\
\dots \quad x_{\si_{k_1+\dots k_{q-1}+1}}=\dots=x_{\si_{k_1+\dots+k_q}},
\end{multline*}
and all other $x_i$ being distinct.
Here $(k_1\ge k_2\ge\dots\ge k_p)$ with $k_1+\dots +k_p\le n$ is the \emph{cycle type} of the permutation $\si$. For our purpose 
we enlarge the cycle type to $(k_1\ge\dots\ge k_p\ge\dots\ge k_q):=(k_1\ge\dots\ge k_p\ge 1\dots\ge 1)$ until it becomes a \emph{partition} of $n$, denoted by 
$$\partition{k}=(k_1\ge\dots\ge k_q) \text{ with } k_1+\dots+k_q = n\,.$$
The conjugate by $\ta\in\S_n$ of the $k_1$-cycle $\si' = (\si_1\, \si_2 \dots \si_{k_1})$ is the $k_1$-cycle 
 $\ta\si'\ta\i = (\ta(\si_1)\, \ta(\si_2) \dots \ta(\si_{k_1}))$, and similarly for the other cycles in $\si$. 
Thus, the isotropy group of any $x$ as above is  conjugated to the subgroup 
$\S_{k_1}\x\S_{k_2}\x\dots\x\S_{k_p}$. 
Its conjugacy class is described by the cycle type $(k_1,\dots,k_p)$ with $k_1,\dots,k_q \in\mathbb N_{>0}$, and equivalently by its enlargement to a partition of $n$.
\end{proof}

The configuration space $M^n$ and the sample space $M^n/\S_n$ are \emph{stratified} by orbit type. 

\begin{definition}[Orbit-type strata]
\nmb.{2.7}
Let $(H)$ denote the conjugacy class of any subgroup $H$ of $\S_n$ corresponding to a partition $\partition{k}$. 
We write $(M^n)_{(H)}$ and $(M^n)_{\partition{k}}$ for the \emph{stratum} of all points in $M^n$ of orbit type $(H)$ and $\partition{k}$, respectively. 
Similarly, we write $(M^n/\S_n)_{(H)}$ and $(M^n/\S_n)_{\partition{k}}$ for the corresponding stratum in $M^n/\S_n$.
\end{definition}

\begin{lemma}[Orbit-type strata]
\nmb.{2.8}
The stratum $(M^n)_{\partition{k}}$ of orbit type $$\partition{k}\coloneqq(k_1\ge\dots\ge k_q)$$ consists of all $x=(x_1,\dots,x_n)$ such that $k_1$ of the the $x_i$ are equal to $y_1\in M$, $k_2$ of the remaining $x_i$ are equal to $y_2\ne y_1$ in $M$, and so on, until the remaining $k_q$ of the $x_i$ are equal to $y_q\in M$, and all $y_i$ are distinct. 
Thus, $(M^n)_{\partition{k}}$ is the disjoint union of its connected components, which are all homeomorphic to the open subset of pairwise distinct points in $M^q$.
\end{lemma}

\begin{proof}
This follows from the description of orbit types in the proof of \nmb!{2.6}. 
\end{proof}

\begin{definition}[Half-ordering of orbit types]
\nmb.{2.9}
For two conjugacy classes $(H)$ and $(H')$ of subgroups $H$ and $H'$ in $\S_n$, we write $(H)\le(H')$ if $H$ is conjugated in $\S_n$ to a subgroup of $H'$. 
Correspondingly, for two partitions $\partition{k}=(k_1\ge\dots\ge k_q)$ and $\partition{k'}=(k'_1\ge\dots\ge k'_{q'})$, we write  $\partition{k}\ge\partition{k'}$ if $\partition{k}$ sub-partitions $\partition{k'}$.
\end{definition}

Note that the half-order between partitions is the inverse of the half-order between the corresponding conjugacy classes.
The \emph{diagonal} $\{x:x_1=\dots=x_n\}$ is the stratum with the largest conjugacy class $(S_n)$ and the smallest partition $(n)$. The projection onto the corresponding stratum in $M^n/\S_n$ is a homeomorphism. 
The \emph{regular stratum} is the open and dense subset of all configurations $x$ with mutually distinct components $x_i$. It has as orbit type the smallest conjugacy class $(\{\Id\})$ and the largest partition $(1\ge\dots\ge 1)$. The regular orbit stratum  $M^n_{(\{\Id\})} = M^n_{(1\ge1\ge\dots\ge1)}$ in $M^n/\S_n$ is open, dense, connected, and locally connected;
it will also be denoted by $M^n_{\text{reg}}$. Likewise for $(M^n/S_n)_{\text{reg}} = (M^n/S_n)_{(1\ge1\ge\dots\ge1)}$.

Note that for $q\le n$, the $q$-skeleton of $M^n/\S_n$ is the the union of all orbit strata $(M^n/\S_n)_{\partition{k}}$ corresponding to all partitions 
$\partition{k}=(k_1\ge\dots \ge k_p)$  with  length $p\le q$. 

\begin{lemma}[Closure of orbit-type strata]
\nmb.{2.10}
The stratum $(M^n)_{\partition{k'}}$ is contained in the closure of the stratum $(M^n)_{\partition{k}}$ if and only if $\partition{k'}\le\partition{k}$ if and only if 
$(\S_{k_1}\x\dots\x\S_{k_q})\le (\S_{k'_1}\x\dots\x\S_{k'_{q'}})$.
Moreover, the closure of $(M^n)_{\partition{k}}$ in $M^n$ is the disjoint union of all $(M^n)_{\partition{k'}}$ with $\partition{k'}\le \partition{k}$.
A similar statement holds with $M^n$ replaced by $M^n/\S_n$.
\end{lemma}

\begin{proof}
This follows from the description of the orbit-type strata given in \nmb!{2.8}, since at the boundary some distinct $x_i$ might become equal. 
\end{proof}

\begin{lemma}[Bundle structure of orbit-type strata]
\nmb.{2.11}
Let $\partition{k}\coloneqq(k_1\geq\dots\geq k_q)$ be a partition describing the orbit type $(H)$ with  $H\coloneqq\S_{k_1}\x\dots\x\S_{k_q}$. 
Then the projection $(M^n)_{\partition{k}}\to \S_n/N_{\S_n}(H)$ defines a topological fiber bundle, 
where $N_{\S_n}(H)$ is the normalizer of $H$ in $\S_n$, and where for any $\si\in\S_n$, the fiber over $\si.N_{\S_n}(H)$ is the fixed-point set $(M^n)^{\si\i H\si}\cap (M^n)_{\partition{k}}$.  
\end{lemma}

\begin{proof}
The proof in \cite[29.22]{Michor08}, although given for smooth manifolds, is purely topological and applies here without change.
\end{proof}

\section{Path metrics on sample spaces}
\nmb0{3}

The category of \emph{path-metric spaces} is ideally suited for the description of sample spaces because it is well-behaved under quotients.
We refer to the appendix for the definition of path metrics and some of their properties, and to the book of Gromov \cite{gromov1999metric} and also \cite{burago2001course} or \cite{AKPetrunin2019} for further details. 
Throughout this section, $d$ is a complete path metric on the separable topological space $M$, $n \in \mathbb N_{>0}$, and $p \in [1,\infty)$. 

There are many choices of metrics on the configuration space $M^n$ which are consistent with the product topology. 
The following lemma describes some of them. 

\begin{lemma}[Path metrics on configuration spaces]
\nmb.{3.1}
The following is a complete path metric on the configuration space $M^n$:
\begin{align*}
d_p(x,y) 
&\coloneqq 
\Big(\frac1n\sum_{i=1}^n d(x_i,y_i)^p\Big)^{1/p},
\qquad
x,y \in M^n.
\end{align*}
The identity on $M^n$ is Lipschitz continuous between any of the metrics $d_p$, $p\in[1,\infty)$.
\end{lemma}

Note that $d_p(x,y)=\|d(x,y)\|_{L^p}$, where $\|\cdot\|_{L^p}$ denotes the $L^p$ norm of functions on the space $\{1,\dots,n\}$ with the uniform probability distribution.
The choice of normalizing constant $\frac1n$ is motivated by this probabilistic interpretation, as well as the large-sample limits in \nmb!{4.3} and \nmb!{4.8}.

\begin{proof}
Completeness of $(M^n,d_p)$ follows from completeness of $(M,d)$.
As $M$ is a path-metric space, there exists by \nmb!{8.3} for any $r>1/2$ and any $a,b \in M$ a point $c=c(a,b)\in M$ such that
\begin{align*}
\max\{d(a,c),d(c,b)\} \leq r d(a,b).
\end{align*}
Then obtains for the configuration $z\coloneqq c(x,y)$ by applying the $L^p$ norm that
\begin{align*}
\max\{d_{p}(x,z),d_{p}(z,y)\} 
\leq
r
d_{p}(x,y).
\end{align*}
This implies by \nmb!{8.3} that $d_p$ is a path metric on $M^n$.
The identity $M^n\to M^n$ is Lipschitz continuous under any of the metrics $d_p$ because 
\begin{equation*}
n^{-1/p} \max_i d(x_i,y_i) \le d_p(x,y) \le \max_i d(x_i,y_i),
\qquad
x,y \in M^n.\qedhere
\end{equation*}
\end{proof}

The complete path metric $d_p$ on the ordered sample space $M^n$ induces a canonical \emph{quotient metric} on the sample space $M^n/\S_n$. 
As the permutation group $\S_n$ acts isometrically on $(M^n,d_p)$, this quotient metric is again complete and admits a particularly simple description, as shown next. 

\begin{lemma}[Quotient metrics on sample spaces]
\nmb.{3.2}
The following quotient metric is a complete path metric on the sample space $M^n/\S_n$:
\begin{alignat*}{3}
\bar d_p(\bar x,\bar y) 
&= 
\min_{\proj(x) = \bar x, \proj(y)=\bar y}d_p(x,y) 
&&= 
\min_{\sigma\in \S_n}d_p(x,y_\sigma),
\end{alignat*}
where $\bar x,\bar y \in M^n/\S_n$ and $x,y \in M^n$ with $\proj(x)=\bar x$, $\proj(y)=\bar y$.
\end{lemma}
 
\begin{proof}
The fibers of the projection are the orbits of the permutation group $\S_n$, which acts isometrically on $(M^n,d_p)$.
Therefore, the metric $\bar d_p$ is a path metric \cite[Lemma 3.3.6]{burago2001course}.
Moreover, this metric is complete: Given a Cauchy sequence in $M^n/\S_n$, take a subsequence such that the distances between subsequent points are summable. Lift the sequence to $M^n$ such that distances between subsequent points are preserved. Then the lift is a Cauchy sequence, which converges thanks to the completeness of $M^n$.
\end{proof}

Recall that a subset of a metric space is called \emph{convex} if the restriction of the metric to this subset is a finite complete path metric \cite[Definition~3.6.5]{burago2001course}.
If the surrounding space carries a complete path metric, then this is equivalent to the subset being \emph{totally geodesic}, i.e., any two points in the subset can be connected by a minimizing geodesic in the subset.  

\begin{lemma}[Convexity of orbit-type strata]
\nmb.{3.3}
Connected components of orbit-type strata in the configuration space $(M^n,d_p)$ have convex closures. Moreover, orbit-type strata in the sample space $(M^n/\S_n,\bar d_p)$ have convex closures.
\end{lemma}

\begin{proof}
If $\partition{k}\coloneqq(k_1\ge\dots\ge k_q)$ is a partition of $n$, then by \nmb!{2.8} each connected component $K$ of $(M^n)_{\partition{k}}$ is homeomorphic to the open subset of all pairwise distinct points in $M^q$.
This homeomorphism is even an isometry (up to a normalizing constant) under the $d_p$ metrics on $M^n$ and $M^q$, respectively.
Thus, the closure $\bar K$ is homeomorphic to $M^q$. 
As $(M^q,d_p)$ is a complete path-metric space by \nmb!{3.1}, it follows that $\bar K$ is a convex subset of $(M^n,d_p)$. 
The projection $\proj\colon M^n\to M^n/\S_n$ restricts to an isometry $\proj\colon K\to (M^n/\S_n)_{\partition{k}}$. It follows that $\bar d_p$ restricts to a complete path metric on the closure of the stratum $(M^n/\S_n)_{\partition{k}}$. Therefore, by definition, the closure of the stratum $(M^n/\S_n)_{\partition{k}}$ is a convex subset of $(M^n/\S_n,\bar d_p)$.
\end{proof}

\begin{example}[Lack of strict convexity]
\nmb.{3.4}
The closure of a connected component of an orbit stratum in $M^n$ need not be \emph{strictly convex} in the sense that \emph{each} minimal geodesic connecting two points in this stratum lies also in the stratum.
\end{example}

\begin{proof}
Let $c_1$ and $c_2$ be two distinct meridian geodesics in the 2-sphere $M=S^2$, which connect the north pole $N$ to the south pole $S$. 
Then $c=(c_1,c_2)$ is a minimizing geodesic between the points $(N,N)$ and $(S,S)$ in $M^2$. 
These points belong to the closed and connected orbit stratum $(M^2)_{(2)}$, but the geodesic $c$ does not lie in $(M^2)_{(2)}$.
\end{proof}

\begin{theorem}[Geodesics between configurations]
\nmb.{3.5}
A continuous curve $c\colon[0,1]\to M^n$ is a constant-speed minimizing geodesic in $(M^n,d_p)$ with $p\in(1,\infty)$ if and only if its component curves $c_1,\dots,c_n\colon[0,1]\to M$ are constant-speed minimal geodesics in $(M,d)$. 
For $p=1$ a similar statement holds without the requirement of constant speed.
\end{theorem}

\begin{proof}
For $p>1$, we associate Lagrangian energy--action pair $(E,A)$ and $(E_p,A_p)$ to $(M,d)$ and $(M^n,d_p)$, respectively, as in \nmb!{8.5}: 
\begin{align*}
E^{s,t}(x_i,y_i)
&\coloneqq
\frac{d(x_i,y_i)^p}{|s-t|^{p-1}},
&
A^{s,t}(c_i) 
&\coloneqq 
\sup_{\substack{n\in\mathbb N\\s=u_0\leq \dots\leq u_n=t}}
\sum_{m=0}^{n-1}
\frac{d(c_i(u_m),c_i(u_{m+1}))^p}{|u_m-u_{m+1}|^{p-1}},
\\
E_p^{s,t}(x,y)
&\coloneqq
\frac{d_p(x,y)^p}{|s-t|^{p-1}},
&
A_p^{s,t}(c) 
&\coloneqq 
\sup_{\substack{n\in\mathbb N\\s=u_0\leq \dots\leq u_n=t}}
\sum_{m=0}^{n-1}
\frac{d_p(c(u_m),c(u_{m+1}))^p}{|u_m-u_{m+1}|^{p-1}},
\end{align*}
for any $i \in \{1,\dots,n\}$, $0\le s\le t\le 1$, $x,y \in M^n$, and continuous curve $c\colon [0,1]\to M^n$.
By \nmb!{8.5}, the given curve $c$ is a length-minimizing constant-speed geodesic in $(M^n,d_p)$ if and only if it satisfies for all $u\leq v\leq w$ in $[0,1]$ that
\begin{equation*}
E_p^{u,v}(c(u),c(v))+E_p^{v,w}(c(v),c(w)-E_p^{u,w}(c(u),c(w))=0.
\end{equation*}
Equivalently, by the definitions of $E$ and $E_p$, 
\begin{equation*}
\frac1n\sum_{i=1}^n \Big(E^{u,v}\big(c_i(u),c_i(v)\big)+ E^{v,w}\big(c_i(v),c_i(w)\big)-E^{u,w}\big(c_i(u),c_i(w)\big)\Big)=0.
\end{equation*}
As all summands are non-negative by the triangle inequality, they vanish.
Equivalently, by Lemma~\nmb!{8.5}, all components $c_i\colon[0,1]\to M$ are constant-speed minimizing geodesics.

For $p=1$, one uses a similar argument for the energy-action pairs $(d,\ell)$ and $(d_1,\ell_1)$, where $\ell$ is the length functional in $(M,d)$, and $\ell_1$ is the length functional in $(M^n,d_1)$. 
However, in this case, a curve is minimizing for these energy-action pairs if and only if it is a geodesic, regardless of whether it has constant speed or not.
\end{proof}

\begin{theorem}[Geodesics between samples]
\nmb.{3.6}
Let $M$ be a connected complete locally compact path-metric space. Then
any minimizing geodesic in the sample space $(M^n/\S_n,\bar d_p)$ is the projection of a minimizing geodesic in the configuration space $(M^n,d_p)$, which we call its {\sl horizontal lift}. 
\end{theorem}

\begin{proof}
Let $\bar c\in C([0,1],M^n/\S_n)$ be a constant-speed minimizing geodesic, 
and let $x\in \proj\i(\bar c(0))$. 
For each $m\in \mathbb N$ we construct a curve $c_m\in C([0,1],M^n)$ as follows:
Set $c_m(0)\coloneqq x$, and then
inductively, for each $k \in \{0,\dots,m-1\}$, choose $c_m|{[(k/m,(k+1)/m]}$ as a constant-speed minimizing geodesic from $c_m(k/m)$ to the orbit $\proj\i(\bar c((k+1)/m))$, until $c_m$ reaches the orbit $\proj\i(\bar c(1))$.
The family $\{c_m:m\in\mathbb N\}$ is equicontinuous because the curves $c_m$ have constant speed. 
Moreover, all curves $c_m$ take values in the compact ball of radius $\bar d_p(\bar c(0),\bar c(1))$ around $x$, which is compact by the Hopf--Rinov theorem \nmb!{8.2}.
Thus, by the Arzel\`a--Ascoli theorem \cite[Theorem~43.15]{willard2004general}, the set $\{c_m:m\in\mathbb N\}$ is pre-compact in the topology of uniform convergence and therefore has a cluster point $c\in C([0,1],M^n)$. 
The cluster point satisfies $\proj\circ c=\bar c$ because the curves $c_m$ satisfy $\proj(c_m(k/m))=\bar c(k/m)$ for all $0\le k\le m$. 
By construction, $c$ is a minimizing geodesic.    
\end{proof}

We next consider the special case where $M$ is a finite-dimensional manifold with Riemannian metric $g$ and complete geodesic distance $d$.
Then $\frac1n(g\oplus\dots\oplus g)$ is an $\S_n$-invariant Riemannian metric on $M^n$, whose geodesic distance is the metric $d_2$ on $M^n$.
The quotient space $M^n/\S_n$ carries a rich differential-geometric structure, which is described in detail in \cite[Sections 29 and 30]{Michor08}.
In particular, one obtains by differential-geometric arguments that a minimal geodesic segment is more regular at interior points than at the end points. 
This is formalized in the following theorem.

\begin{theorem}[Interior regularity of Riemannian geodesics {\cite[3.5 and 3.4]{Michor03orbit}}]
\nmb.{3.7}
Let $M$ be a finite-dimensional manifold with complete Riemannian metric $g$, and let $M^n$ be the product manifold with the product Riemannian metric $\frac1n(g\oplus\dots\oplus g)$. 
Then, for any lift $c\colon[0,1]\to M^n$ of a minimal geodesic segment in $M^n/\S_n$, the isotropy group $(\S_n)_{c(t)}$ of an interior point of $c$ is contained in the isotropy groups $(\S_n)_{c(0)}$, $(\S_n)_{c(1)}$ of the end points.
\end{theorem}

Thus, for any subgroup $H\leq\S_n$, the set $(M^n/\S_n)_{\le(H)}$ of orbits with orbit type smaller or equal to $(H)$ is a strictly convex subset of $M^n/\S_n$.
This means that any minimal geodesic segment between two points in $(M^n/\S_n)_{\le(H)}$ lies entirely in $(M^n/\S_n)_{\le(H)}$.
In particular, the regular orbit-type stratum in $M^n/\S_n$ is a strictly convex open dense subset.
Recall for comparison that $(M^n/\S_n)_{\ge(H)}$ is convex by \nmb!{3.3} but may not be strictly convex by \nmb!{3.4}.

\begin{example}[Lack of interior regularity]
\nmb.{3.8}
The assertion of \nmb!{3.7} is wrong for non-Riemannian complete path-metric spaces.
\end{example}

\begin{proof}
Let $(M,d)$ be an open book space, for example the 3-spider, one of the simplest tree spaces \cite{BHV01}.
$$
{\hbox{
\psfrag x {$x$}
\psfrag y {$y$}
\psfrag z {$z$}
\psfrag 0 0
\includegraphics*[height=1.9cm]{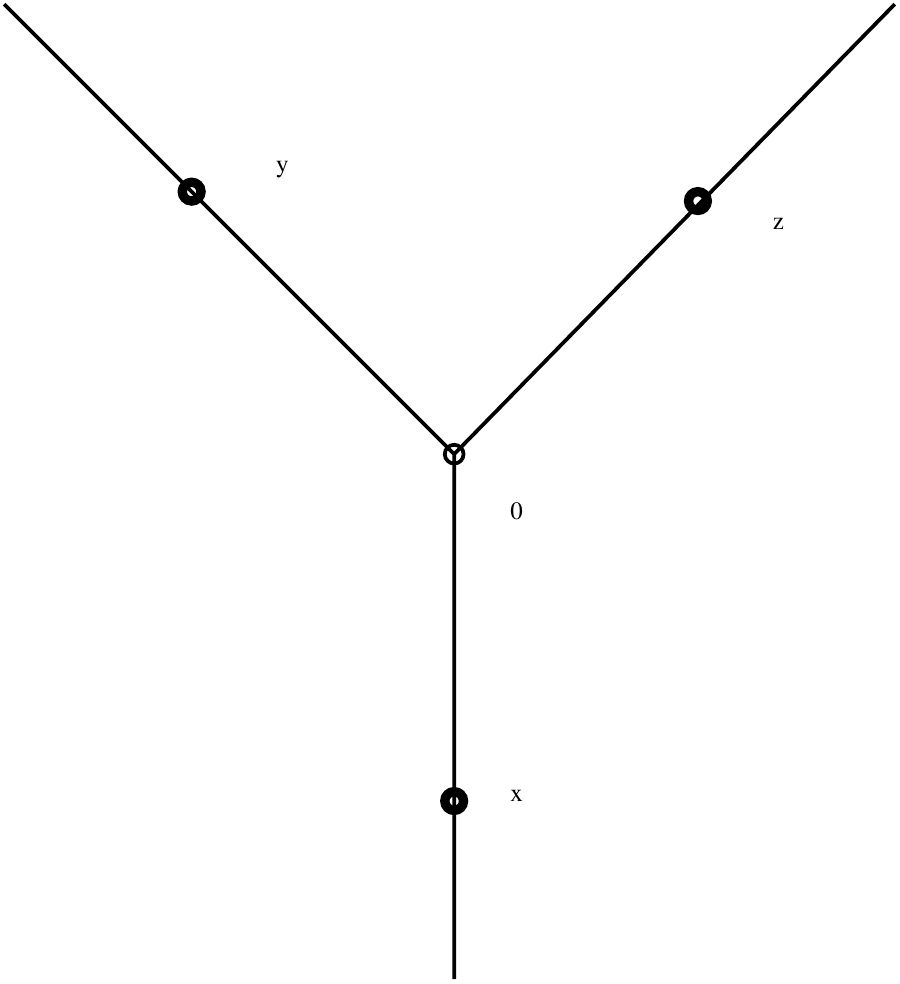}}}\qquad\qquad  
{\begin{minipage}[b][2.2cm][s]{0.7\textwidth}
We choose 3 points $x,y,z$ on the 3 lines with the same distance from the center $0$.
Let $c\colon[0,2]\to M^2$ be the minimal geodesic from $c(0)=(x,y)$ via $c(1)=(0,0)$ to $c(2)=(z,z)$.
Then the isotropy group $\S_2$ of $c(1)$ and $c(2)$ is not contained in the trivial isotropy group of $c(0)=(x,y)$. 
\end{minipage}}
$$
See the related discussion in \cite[Chapter 8]{villani2008optimal}. Note that the `curvature' of the spider at 0 is $-\infty$.
\end{proof}

\section{Infinite configuration and sample spaces}
\nmb0{4}

This section exhibits configuration spaces as spaces of random variables and sample spaces as spaces of probability distributions. 
Moreover, it identifies large-sample limits of these spaces.
Throughout this section, $(M,d)$ is a separable connected complete path-metric space,
and $p \in [1,\infty)$.

\begin{definition}[Random variables]
\nmb.{4.1}
For any complete probability space $(\Omega,\mathcal F,\mathbb P)$, we write $L^p(\Omega,M)$ for the space of all measurable functions $x\colon \Omega\to M$ which satisfy for one (or equivalently, all) $o \in M$ that $\|d(x,o)\|_{L^p(\Omega)}<\infty$.
We endow the space $L^p(\Omega,M)$ with the metric 
\begin{align*}
d_p(x,y) \coloneqq \|d(x,y)\|_{L^p(\Omega)},
\qquad
x,y \in L^p(\Omega,M).
\end{align*} 
\end{definition}

\begin{lemma}[Configurations as random variables]
\nmb.{4.2}
For any $n \in \mathbb N$, the configuration space $(M^n,d_p)$ is isometric to $(L^p(\{1,\dots,n\},M),d_p)$, where $\{1,\dots,n\}$ is seen as a probability space with the uniform distribution.
\end{lemma}

\begin{proof}
A configuration $x\in M^n$ is precisely a function $x\colon\{1,\dots,n\}\to M$, and the metrics $d_p$ defined on $M^n$ and $L^p(\{1,\dots,n\},M)$ coincide. 
\end{proof}

The description of configurations as random variables allows one to pass to a \emph{large-sample limit}.
Similar results are shown in \cite{kuwae2008variational}.
In the following lemma, $(0,1)$ denotes the unit interval with the Lebesgue measure and could, for all purposes, be replaced by any standard probability space.

\begin{lemma}[Infinite configurations]
\nmb.{4.3}
The configuration spaces $(M^n,d_p)$ are isometrically embedded in the complete path-metric space $(L^p((0,1),M),d_p)$ and converge to it in the following sense:
for any compact $K\subset L^p((0,1),M)$,
\begin{equation*}
\lim_{n\to\infty}\sup_{x\in K}\inf_{y\in M^n}d_p(x,y)=0.
\end{equation*}
\end{lemma}

The lemma would imply pointed Gromov--Hausdorff convergence of $(M^n,d_p)$ to the space $L^p((0,1),M)$ if the uniform convergence on compacts could be strengthened to uniform convergence on bounded sets. 
However, this is not the case, as one easily verifies by considering functions $x$ of the form $n^{1/p}\mathbbm{1}_{[0,1/n]}$ for large $n$.

\begin{proof}
The isometric immersion of $(M^n,d_p)\cong L^p(\{1,\dots,n\},M)$ into $L^p((0,1),M)$ is given by the identification of $n$-tuples with piece-wise constant functions on $(0,1)$.
It remains to prove the convergence. 
Let $\epsilon>0$.
By the compactness of $K$, there are $m \in \mathbb N$ and $x_1,\dots,x_m \in L^p((0,1),M)$ such that the open $d_p$-balls $B_{\epsilon/3}(x_i)$ cover $K$.
Let $o \in M$.
By the dominated convergence theorem, there is $r>0$ such that the configurations $y_i \in L^p((0,1),M)$ defined by
\begin{align*}
y_i
\coloneqq 
\left\{
\begin{aligned}
&x_i,&&d(x_i,o)\leq r,
\\
&o,&&d(x_i,o)>r,
\end{aligned}
\right.
\end{align*}
satisfy $d_p(x_i,y_i)\leq\epsilon/3$ for all $i\in\{1,\dots,m\}$.
Let $F$ be the Banach space of continuous bounded functions on $B_r(o)$ with the uniform norm. 
Then $(B_r(o),d)$ embeds isometrically into $F$ via the map $B_r(o) \ni a \mapsto d(a,\cdot) \in F$.
Thus, $B_r(o)$ may be seen as a subset of $F$. 
Moreover, $F$ is separable because $B_r(o)$ is separable. 
For any $n \in \mathbb N$, let $\mathbb E_n\colon L^p((0,1),F)\to L^p((0,1),F)$ be the conditional expectation with respect to the sigma-algebra generated by the intervals $[\frac{j-1}{n},\frac{j}{n})$, $j \in \{1,\dots,n\}$.
Then, for sufficiently large $n$, the configurations $z_i\coloneqq \mathbb E_n(y_i)$ satisfy $d_p(y_i,z_i)\leq\epsilon/3$ for all $i \in \{1,\dots,m\}$.
Let $A\colon F \to M$ be the metric projection from $f \in F$ to the nearest point $A(f) \in M$, 
and let $A_*\colon L^p((0,1),F)\to L^p((0,1),M)$ be the push-forward along $A$. 
Then the configurations $w_i\coloneqq A_*z_i$ satisfy for all $i \in \{1,\dots,m\}$ that
\begin{align*}
d_p(z_i,w_i)
=
d_p(z_i,A_*z_i)
\leq
d_p(z_i,y_i)
\leq 
\epsilon/3.
\end{align*}
It follows that every $x \in K$ is $\epsilon$-close to some $w_i \in L^p(\{1,\dots,n\},M)$.
\end{proof}

Recall that any continuous curve $c\colon[0,1]\to L^p((0,1),M)$ has a jointly measurable version $c\colon[0,1]\times(0,1)\to M$; see e.g.\@ \cite[Proposition~3.2]{daprato2014stochastic}. 
Then the \emph{sample paths} of $c$ are the measurable functions $c(\cdot,\omega)\colon[0,1]\to M$, $\omega \in (0,1)$.

\begin{lemma}[Geodesics between infinite configurations]
\nmb.{4.4}
$(L^p((0,1),M),d_p)$ is a complete path-metric space.
For $p>1$, a continuous curve $c\colon[0,1]\to L^p((0,1),M)$ is a constant-speed minimizing geodesic in $(L^p((0,1),M),d_p)$ if and only if almost all of its sample paths are constant-speed minimizing geodesics in $M$. 
\end{lemma}

\begin{proof}
To show that $L^p(\Omega,M)$ is a complete path-metric space, we proceed as in the proof of \nmb!{3.1}, 
noting that the point $c=c(a,b)$ can be chosen as a measurable function of $a,b$.
Indeed, this follows from a measurable selection theorem \cite{dellacherie1975ensembles} because the set 
\begin{equation*}
\Gamma\coloneqq \big\{(a,c,b)\in M^3:\max\{d(a,c),d(c,b)\}\leq \al d(a,b)\big\} 
\end{equation*}
is Polish, the projection $\Gamma\ni(a,c,b)\to (a,b)\in M^2$ is continuous, and the inverse image of any $(a,b)\in M^2$ under this projection is compact. 
To prove the statement about geodesics, we proceed as in the proof of \nmb!{3.5} and associate Lagrangian energy--action pairs $(E,A)$ and $(E_p,A_p)$ to $(M,d)$ and $(L^p((0,1),M),d^p)$, respectively.
By \nmb!{8.3} a continuous curve $c\colon[0,1]\to L^p((0,1),M)$ is a length-minimizing constant-speed geodesic if and only if it satisfies for all $u\leq v\leq w$ in $[0,1]$ that
\begin{equation*}
E_p^{u,v}(c(u),c(v))+E_p^{v,w}(c(v),c(w)-E_p^{u,w}(c(u),c(w))=0.
\end{equation*}
Equivalently, by the definitions of $E$ and $E_p$, 
\begin{equation*}
\mathbb E\big[E^{u,v}(c(u),c(v))+E^{v,w}(c(v),c(w)-E^{u,w}(c(u),c(w))\big]=0,
\end{equation*}
where $\mathbb E$ is the expectation with respect to the Lebesgue measure on $(0,1)$.
Equivalently, the following property holds almost surely: for all rational numbers $u\leq v\leq w$ in $[0,1]$,
\begin{equation*}
E^{u,v}(c(u),c(v))+E^{v,w}(c(v),c(w)-E^{u,w}(c(u),c(w))=0.
\end{equation*} 
By \nmb!{8.3} this implies for almost every $\omega \in (0,1)$ that the sample path 
\begin{equation*}
[0,1]\cap\mathbb Q \ni u \mapsto c(u,\omega)
\end{equation*}
is parameterized by constant speed. 
In particular, any such sample path can be extended continuously to all real numbers in $[0,1]$. 
Thus, we have established that $c$ has a version whose sample paths are almost surely constant-speed minimizing geodesics.
Moreover, this property is equivalent to the previous ones.
\end{proof}

On finite probability spaces, the statement about geodesics in \nmb!{4.4} extends to $p=1$ if the constant-speed condition is omitted, as shown in \nmb!{3.5}.
However, this is not the case on infinite probability spaces, as the following example shows.

\begin{example}[Discontinuity of sample paths]
\nmb.{4.5}
Constant-speed minimizing geodesics in $L^1((0,1),M)$ may have discontinuous sample paths.
\end{example}

\begin{proof}
Let $M=\mathbb R$. The curve
\begin{equation*}
c\colon[0,1]\times(0,1) \to M, 
\qquad
c(t,\omega) \coloneqq \mathbbm{1}_{[t,1]}(\omega)
\end{equation*}
is a constant-speed minimizing geodesic in $L^1((0,1),M)$, but none of its sample paths are continuous. 
\end{proof}

\begin{definition}[probability distributions]
\nmb.{4.6}
Let $\probspace^p(M)$ denote the space of all probability distributions $P$ on $M$ which satisfy for one (equivalently, all) $o \in M$ that $\|d(o,\cdot)\|_{L^p(P)}<\infty$. 
We endow $\probspace^p(M)$ with the \emph{Wasserstein metric},
\begin{equation*}
\bar d_p(P,Q) 
= \inf_{R} \|d(\cdot,\cdot)\|_{L^p(R)},
\quad
P,Q \in \probspace^p(M),
\end{equation*}
where the infimum is 
over all probability distributions $R$ on $M\times M$ with marginals $P,Q$.
Moreover, we write $\probspace_n(M)$ for the subset of all \emph{atomic probability distributions} of the form $\frac1n\sum_{i=1}^n\delta_{x_i}$, where $\delta_{x_i}$ is the Dirac measure centered at $x_i\in M$.
\end{definition}

As an aside, the set $\probspace_n(M)$ of atomic distributions can equivalently be characterized as the set of $\{0,1/n,\dots,1\}$-valued probability measures. 
This equivalence uses the separability of $M$ and is shown in \nmb!{8.6}. 
The following lemma identifies samples with probability distributions, namely, with their \emph{empirical laws}.

\begin{lemma}[Samples as probability distributions]
\nmb.{4.7}
For any $n \in \mathbb N$, the sample space $(M^n/\S_n,\bar d_p)$ is isometric to the space $(\probspace_n(M),\bar d_p)$ of atomic probability distributions.  
\end{lemma}

\begin{proof}
Samples $\bar x=\proj(x) \in M^n/\S_n$ are naturally identified with atomic probability distributions $P=\frac1n\sum_{i=1}^n \delta_{x_i} \in \probspace_n(M)$.
If $\bar y=\proj(y)\in M^n/\S_n$ is another sample with corresponding probability distribution $Q=\frac1n\sum_{i=1}^n\delta_{y_i}\in\probspace_n(M)$, then
\begin{align*}
\bar d_p(\bar x,\bar y)
&=
\min_{\proj(x) = \bar x, \proj(y)=\bar y}d_p(x,y) 
=
\min_{\proj(x) = \bar x, \proj(y)=\bar y}\|d(x,y)\|_{L^p(\{1,\dots,n\})}
\\&= 
\min_{R}\|d(\cdot,\cdot)\|_{L^p(R)},
\end{align*}
where the last minimum is over all atomic probability distributions $R\in \probspace_n(M\times M)$ with marginal laws $P$ and $Q$.
By Birkhoff's theorem, one may equivalently take the minimum over the larger set of all (not necessarily atomic) probability distributions $R$ on $M\times M$ with marginal laws $P$ and $Q$ 
\cite[Proposition~1.3.1]{panaretos2020invitation}.
This shows that the right-hand side equals $\bar d_p(P,Q)$.
Therefore, the identification of samples with probability distributions is an isometry.
\end{proof}

\begin{lemma}[Infinite samples]
\nmb.{4.8}
The sample spaces $(M^n/\S_n,\bar d_p)$ are isometrically embedded in the complete path-metric space $(\probspace^p(M),\bar d_p)$.
For $M$ locally compact, they converge to $(\probspace^p(M),\bar d_p)$ in the following sense: for any compact $K\subset \probspace^p(M)$,
\begin{equation*}
\lim_{n\to\infty}\sup_{P \in K}\inf_{Q \in M^n/\S_n} \bar d_p(P,Q)=0.
\end{equation*}
Here $M^n/\S_n$ is identified with the subset $\probspace_n(M)$ of $\probspace^p(M)$ using \nmb!{4.7}.
\end{lemma}

\begin{proof}
The sample space $(M^n/\S_n,\bar d_p)$ is isometrically embedded in $(\probspace^p(M),\bar d_p)$ as a consequence of \nmb!{4.7}.
It is well-known that the Wasserstein metric $\bar d_p$ on $\probspace_p(M)$ is a complete path metric \cite[Theorem~6.18 and Corollary~7.22]{villani2008optimal}.
It remains to prove the convergence. 
Let $\epsilon>0$. 
As $K$ is compact, there are $m \in \mathbb N$ and $P_1,\dots,P_m \in K$ such that the open $\bar d_p$-balls $B_{\epsilon/2}(P_i)$ cover $K$.
For each $i \in \{1,\dots,m\}$, the empirical distributions of $P_i$ converge to $P_i$ in the Wasserstein distance $\bar d_p$ \cite[Proposition~2.2.6]{panaretos2020invitation}. 
Therefore, there are distributions $Q_1,\dots,Q_m\in \probspace_n(M)$ for some $n \in \mathbb N$ such that $\bar d_p(P_i,Q_i)\leq \epsilon/2$ for all $i \in \{1,\dots,m\}$.
It follows that every $P \in K$ is $\epsilon$-close to some distribution in $\probspace_n(M)$.
\end{proof}

Recall from \nmb!{3.2} that the sample space $(M^n/\S_n,\bar d_p)$ is the path-metric quotient of the configuration space $(M^n,d_p)$ with respect to the action of permutation group of $\{1,\dots,n\}$.
A similar statement applies to infinite sample and configuration spaces, as shown in the following lemma.
In analogy to \nmb!{3.2}, let $\proj\colon L^p((0,1),M)\to \probspace^p(M)$ be the map from random variables to their law or, in more analytic terms, the push-forward of the Lebesgue measure along the given measurable function.
Moreover, let $\Aut((0,1))$ be the automorphism group of the probability space $(0,1)$, i.e., the group of bi-measurable measure-preserving functions from $(0,1)$ to itself.

\begin{lemma}[Quotient structure]
\nmb.{4.9}
The Wasserstein metric $\bar d_p$ on $\probspace^p(M)$ is a quotient metric:
\begin{align*}
\bar d_p(P,Q)
=
\inf_{\proj(x)=P,\proj(y)=Q}d_p(x,y)=\inf_{\sigma \in \Aut((0,1))}d_p(x,y\o\sigma),
\end{align*}
where $P,Q \in \probspace^p(M)$ and $x,y \in L^p((0,1),M)$ with $\proj(x)=P$, $\proj(y)=Q$.
\end{lemma}

\begin{proof}
The first equality holds because any coupling $R$ in the definition \nmb!{4.6} of the Wasserstein metric is the joint law of some random variables $x,y \in L^p((0,1),M)$.
The second equality holds because the action of $\Aut((0,1))$ is nearly transitive on the fibers of $\proj$ in the following sense \cite[Lemma~6.4]{cardaliaguet2013notes}:
for all $x,y\in L^p((0,1),M)$ with $\proj(x)=\proj(y)$ and all $\epsilon>0$, there exists $\sigma \in \Aut((0,1))$ such that $d_p(x,y\circ\sigma)\leq \epsilon$.
\end{proof}

The following lemma generalizes \nmb!{3.6} from finite to infinite configurations and samples, respectively. 

\begin{theorem}[Geodesics between infinite samples]
\nmb.{4.10}
Let $M$ be a connected complete locally compact path-metric space. Then any minimizing geodesic in the infinite sample space $(\probspace^p(M),\bar d_p)$ is the projection of a minimizing geodesic in the configuration space $(L^p(\Omega,M),d_p)$, which we call its {\sl horizontal lift}. 
\end{theorem}

\begin{proof}
This is proven in \cite[Corollary~7.22]{villani2008optimal} along the same lines as \nmb!{3.6}, i.e., using Lagrangian energy-action pairs. The horizontal lift is called displacement interpolation there.
\end{proof}

Skeleta and orbit-type strata of finite sample spaces $M^n/\S_n$ were defined in \nmb!{2.1} and \nmb!{2.7}, respectively.
Via the isometry \nmb!{4.7} to atomic probability distributions and the isometric embedding \nmb!{4.8} into $p$-integrable probability distributions, one obtains straight-forward extensions to skeleta and orbit-type strata of infinite sample spaces, as defined next.

\begin{definition}[Infinite skeleta and orbit-type strata]
\nmb.{4.11}
For any $q \in \mathbb N$, the \emph{$q$-skeleton} in the infinite-sample space $\probspace^p(M)$ is the subset $\probspace(M)_q$ of all probability distributions whose support is a set of at most $q$ points.
Similarly, for any partition $\partition{w} \coloneqq (w_1\geq\cdots\geq w_q)$ of 1 consisting of non-negative real numbers $w_i$ summing up to $1$, the \emph{$\partition{w}$-stratum} in the infinite-sample space $\probspace^p(M)$ is the subset of all $P=\sum_{i=1}^q w_i \delta_{x_i} \in \probspace(M)_q$ with distinct points $x_i$.
The measure $P$ is called \emph{regular} if the points $x_i$ are distinct and the weights $w_i$ are strictly positive.
\end{definition}

\section{Means and polymeans}\nmb0{5}

In this section, we generalize Fr\'echet means \cite{frechet_les_1948} and $k$-means \cite{macqueen_methods_1967} to \emph{polymeans} using the path-metric structure of sample space. 
Background and further references on Fr\'echet means can be found in the textbook \cite{pennec_sommer_fletcher_2020}.
Throughout this section, we consider the configuration space $(M^n,d_p)$ and sample space $(M^n/\S_n,\bar d_p)$ of a connected complete path-metric space $(M,d)$ for some $n \in \mathbb N$ and $p \in [1,\infty)$.
The following definition introduces polymeans as metric projections onto certain subsets of sample space $M^n/\S_n$, namely $q$-skeleta $(M^n/\S_n)_q$ (see \nmb!{2.2}) or $\partition{k}$-strata $(M^n/\S_n)_{\partition{k}}$ (see \nmb!{2.8}).

\begin{definition}[Polymeans]
\nmb.{5.1}
For any $q \in \mathbb N$, a \emph{$q$-mean} of a sample is a $\bar d_p$-nearest point in the $q$-skeleton of sample space. 
Similarly, for any partition $\partition{k}$ of $n$, a \emph{$\partition{k}$-mean} of a sample is a $\bar d_p$-nearest point in the closure of the $\partition{k}$-stratum. 
\end{definition}

Recall that the $q$-skeleton is closed, and the closure of the $\partition{k}$-stratum is the union of all $\partition{k'}$-strata with $\partition{k'}\leq\partition{k}$.
This ensures the existence of $q$-means and $\partition{k}$-means, as shown next. 
One should be aware that a $q$-mean might consist of less than $q$ distinct points, and similarly a $\partition{k}$-mean might have orbit type $\partition{k'}$ with $\partition{k'}\leq\partition{k}$.

\begin{lemma}[Existence of polymeans]
\nmb.{5.2}
If $M$ is a complete locally compact path-metric space, 
then every sample $\bar x\in M^n/\S_n$ has a $q$-mean and a $\partition{k}$-mean, for each $q \in \mathbb N_{>0}$ and orbit type $\partition{k}\coloneqq(k_1\ge \dots\ge k_q)$.
\end{lemma}

\begin{proof} 
For sufficiently large $r>0$, the closed ball $B_r(\bar x)$ has non-empty intersection with the $q$-skeleton. 
By the Hopf--Rinow theorem \nmb!{8.2}, this intersection is compact and therefore contains a point of minimal $\bar d_p$-distance to $\bar x$. 
The argument for the $\partition{k}$-stratum is similar.
\end{proof}

Generic configurations have unique polymeans, as shown next. 
Here generic is understood in a measure-theoretic sense, i.e., up to null sets with respect to a given Riemannian volume form.

\begin{lemma}[Uniqueness of polymeans]
\nmb.{5.3}
Let $M$ be a complete finite-dimensional Riemannian manifold, and assume that $p=2$.
Then the configurations $x\in M^n$ such that $\proj(x)$ has more than one $q$-mean or more than one $\partition{k}$-mean are a null set with respect to the Riemannian volume form. 
\end{lemma}

\begin{proof}
We consider $M^n$ as a complete Riemannian manifold with Riemannian distance $d_2$. 
Let $K$ be the $q$-skeleton or the $\partition{k}$-stratum in $M^n$, 
and let $C$ be the set of all points in $M^n$ whose distance to $K$ is realized by more than one geodesic (sometimes called the medial axis). 
At any point in $C$, the squared distance function to $K$ is non-differentiable \cite[Remark~3.6]{mantegazza2002hamilton}.
These points of non-differentiability constitute a $C^2$-rectifiable set \cite[Proposition~3.7]{mantegazza2002hamilton}.
Thus, its subset $C$ has vanishing measure. 
\end{proof}

We next show that the definition of polymeans extends the definition of Fr\'echet $p$-means.

\begin{example}[Fr\'echet means]
\nmb.{5.4}
Fr\'echet means correspond exactly to $1$-means or, equivalently, $(n)$-means, where $(n)$ denotes the trivial partition. 
\end{example}

\begin{proof}
Recall that the $1$-skeleton in sample space $M^n/\S_n$ consists of all $\bar y = \proj(y,\dots,y)$ with $y \in M$ and coincides with the orbit-type stratum $(M^n/\S_n)_{(n)}$, where $(n)$ denotes the partition of $n$ of length $1$. 
Thus, $1$-means coincide with $(n)$-means and minimize, for a given $\bar x=\proj(x)$ in $M^n/\S_n$, the functional
\begin{align*}
\bar d_p(\bar x,\bar y)
=
\left(\frac1n\sum_{i=1}^n d(x_i,y)^p\right)^{1/p}
\end{align*}
over all $\bar y=\proj(y,\dots,y)$ in the $1$-skeleton of $M^n/\S_n$. 
Minimizers of the right-hand side, seen as a function of $y \in M$, are exactly Fr\'echet means.
Thus, a point $y \in M$ is a Fr\'echet mean of a configuration $x \in M^n$ if and only if the sample $\proj((y,\dots,y))\in M^n/\S_n$ is a $1$-mean, or equivalently an $(n)$-mean, of $\proj(x)\in M^n/\S_n$.
\end{proof}

$k$-mean clustering remains a very popular method in cluster analysis, more than 60 years after \cite{macqueen_methods_1967,jain_data_2010}. Like the Fr\'echet $p$-mean, it can be generalized with the power $p$ of the distance \cite{xu_power_2019}. We show below that this corresponds to our geometric definition of polymeans. 

\begin{example}[$k$-means]
\nmb.{5.5}
$q$-means correspond exactly to $k$-means clustering for $k=q\in\mathbb N$.
\end{example}

\begin{proof}
Let $\bar x, \bar y \in M^n/\S_n$ with $\bar y$ belonging to the $q$-skeleton. 
Then there are lifts $x,y \in M^n$ such that $\proj(x)=\bar x$, $\proj(y)=\bar y$, and $d_p(x,y)=\bar d_p(\bar x,\bar y)$.
The set $\{1,\dots,n\}$ can be partitioned into non-empty subsets $A_1,\dots,A_q$ such that $y_i=y_j$ for any $i,j\in S_k$ and $k \in \{1,\dots,q\}$.
Then
\begin{align*}
n \bar d_p(\bar x,\bar y)^p
=
\sum_{i=1}^q 
\sum_{j\in A_i} d(x_j,y_i)^p.
\end{align*}
The left-hand side is minimized by $q$-means $\bar y$, and the right-hand side is minimized by partitions $A_1,\dots,A_k$ and $k$-means $(y_1,\dots,y_k)$ with $k=q$.
Therefore, the $q$-mean and $k$-mean problems are equivalent. 
As an aside, the $q$-mean vector $\bar y$ does not encode the optimal correspondence between points $x_i$ and $y_i$, and the $k$-mean vector $(y_1,\dots,y_k)$ does not encode the multiplicities $\#A_i$. 
However, this information can be retrieved easily by matching each point $x_j$ to the nearest point $y_i$.
\end{proof}

\begin{definition}[Clusters]
\nmb.{5.6}
A \emph{clustering} of a sample $\bar x \in M^n/\S_n$ is a representation $\bar x = \bar x_1 \sqcup \cdots \sqcup \bar x_q\coloneqq\proj((x_1,\dots,x_q))$, where $\bar x_i=\proj(x_i) \in M^{k_i}/\S_{k_i}$ for some partition $k_1+\dots+k_q=n$ with $k_i \in \mathbb N_{>0}$ and $q \in \mathbb N_{>0}$.
In this situation, $\bar x_i$ are called \emph{clusters} or \emph{sub-samples} of sizes $k_i$.
\end{definition}

\begin{lemma}[Polymeans as clusters]
\nmb.{5.7}
If $\bar y$ is a $q$-mean of $\bar x$, then there are clusterings $\bar x=\bar x_1\sqcup\cdots\sqcup\bar x_q$ and $\bar y=\bar y_1\sqcup\cdots\sqcup\bar y_q$ such that each $\bar y_i$ is a $1$-mean of $\bar x_i$.
Moreover, if $\bar y$ is a $\partition{k}$-mean of $\bar x$ with $\partition{k}\coloneqq(k_1\geq\dots\geq k_q)$, then the partition can be chosen such that each cluster $\bar x_i$ has size $k_i$.
\end{lemma}

\begin{proof}
Let $A_1,\dots,A_q$ be a partition of $\{1,\dots,n\}$ as in the proof of \nmb!{5.6}.
Then the clusterings $\bar x_i=\proj((x_j)_{j\in A_i})$ and $\bar y_i=\proj((y_j)_{j\in A_i})$ have the desired property.
\end{proof}

Lemma~\nmb!{5.7} exhibits polymeans  as \emph{weighted means}, where the weights correspond to the cluster sizes, normalized by the total number of samples. 
The same interpretation is obtained by identifying polymeans with atomic measures via \nmb!{4.7}.
In some situations it may be advantageous to consider \emph{unweighted polymeans}, which encode only the locations but not the weights of the clusters.
The following definition describes $q$ such clusters located at mutually distinct points $y_1,\dots,y_q \in M$.
Recall that the ensemble of such mutually distinct point configurations modulo permutations is the regular stratum $(M^q/\S_q)_{\text{reg}}$.

\begin{definition}[Unweighted $q$-means]
\nmb.{5.8}
For any $q\in\mathbb N$, an \emph{unweighted $q$-mean} of a sample $\bar x=\proj(x) \in M^n/\S_n$ is a regular $q$-sample $\bar z \in (M^q/\S_q)_{\text{reg}}$ which minimizes the functional
\begin{equation*}
(M^q/\S_q)_{\text{reg}}
\ni 
\bar z = \proj(z)
\mapsto
\sum_{i=1}^n \min_{j\in\{1,\dots,q\}} d(x_i,z_j)^p.
\end{equation*}
\end{definition}

Unweighted $q$-means may fail to exist for a given $q\in\mathbb N_{>0}$ because the regular stratum $(M^q/\S_q)_{\text{reg}}$ is not closed.  
It is, however, open and dense.  
Thus, for any given $q \in \mathbb N_{>0}$, there always exists an unweighted $q'$-mean with $q'\leq q$.
The definitions of weighted and unweighted polymeans are consistent with each other in the following sense.

\begin{lemma}[Relation between weighted and unweighted $q$-means]
\label{5.relation}
Let $\bar x \in M^n/\S_n$, 
and let $z_1,\dots,z_q$ be distinct points in $M$. 
Then $\proj(z_1,\dots,z_q) \in (M^q/\S_q)_{\text{reg}}$ is an unweighted $q$-mean of $\bar x$ if and only if
\begin{equation*}
\proj((\underbrace{z_1,\dots,z_1}_{k_1\text{ times}},\dots,\underbrace{z_q,\dots,z_q}_{k_q\text{ times}})) \in M^n/\S_n
\end{equation*}
is a $q$-mean of $\bar x$ for some integer weights $k_i$ summing up to $n$.
\end{lemma}

\begin{proof}
This easily follows from the definitions. 
\end{proof}

Skeleta and orbit-type strata in infinite sample space $\probspace^p(M)$ were defined in \nmb!{4.10}.
This yields the following straight-forward extensions to polymeans of infinite samples.

\begin{definition}[Population polymeans]
\nmb.{5.9}
A \emph{population $q$-mean} of an infinite sample $P \in \probspace^p(M)$ is a $\bar d_p$-nearest point in the $q$-skeleton of $\probspace^p(M)$. 
Similarly, for any partition $\partition{k} \coloneqq (k_1\geq\cdots\geq k_q)$ consisting of non-negative real numbers $k_i$ summing up to $1$, a \emph{population $\partition{k}$-mean} of $P \in \probspace^p(M)$ is a $\bar d_p$-closest point in the $\partition{k}$-stratum of $\probspace^p(M)$.
Moreover, an \emph{unweighted population $q$-mean} of $P \in \probspace^p(M)$ is a $\bar d_p$-closest point in the regular stratum of $\probspace_q(M)$.
\end{definition}

\section{Random samples}
\nmb0{6}
\label{sec:probability_distributions}

Throughout this section, we consider the configuration space $(M^n,d_p)$ and sample space $(M^n/\S_n,\bar d_p)$ of a separable complete path-metric space $(M,d)$ for some $n \in \mathbb N$ and $p \in [1,\infty)$.
We use the letter $\mathcal P$ to designate probability distributions. 
Thus, $\probspace(M^n/\S_n)$ is the set of probability distributions on sample space, 
and $\probspace(M^n)$ is the set of all probability distributions on configuration space. 
Moreover, we write $\probspace(M^n)_{S_n}$ for the subset of \emph{symmetric} probability distributions, where symmetry means $\S_n$-invariance.   

\begin{lemma}[Distributions of samples]
\nmb.{6.1}
Probability distributions on sample space $M^n/\S_n$ correspond exactly to symmetric probability distributions on configuration space $M^n$. 
\end{lemma}

\begin{proof}
We claim that the projection from configuration onto sample space induces a bijection
\begin{equation*}
\probspace(M^n)_{\S_n}\ni P \mapsto \proj_*P \in \probspace(M^n/\S_n).
\end{equation*}
To prove the claim, we will construct an inverse of this map by randomization over the $\S_n$-orbit using the probability kernel
\begin{align*}
K\colon M^n\ni x\mapsto\frac{1}{n!}\sum_{\sigma \in \S_n}\delta_{x_\sigma} \in \probspace(M^n)_{\S_n}.
\end{align*}
This kernel is $\S_n$-invariant and consequently descends to a probability kernel
\begin{equation}
\tag{1}
\bar K\colon M^n/\S_n\ni \bar x=\proj(x)\mapsto\frac{1}{n!}\sum_{\sigma \in \S_n}\delta_{x_\sigma} \in \probspace(M^n)_{\S_n},
\end{equation}
which maps samples $\bar x$ to uniform distributions on their fibers $\proj^{-1}(x)$ in configuration space. 
The two kernels are related by $K=\bar K\o\proj$.
For any probability distribution $\bar P$ on $M^n/\S_n$, we write $\int \bar K(\bar x)\bar P(\d\bar x)$ for the composition of the kernel $\bar K$ with the probability distribution $\bar P$. 
Formally, this is a measure-valued Pettis integral.
Then the map
\begin{equation}
\tag{2}
\probspace(M^n/\S_n) \ni \bar P \mapsto \int \bar K(\bar x)\bar P(\d\bar x)\in \probspace(M^n)_{S_n}
\end{equation}
is an inverse to the map $\proj_*$ because
\begin{align*}
\proj_*\int\bar K(\bar x)\bar P(\d\bar x)
&=
\int \proj_*\big(\bar K(\bar x)\big)\bar P(\d\bar x)
=
\int \delta_{\bar x}\bar P(\d\bar x)
=
\bar P, 
\\
\int\bar K(\bar x)(\proj_*P)(\d\bar x)
&=
\int\bar K\big(\proj(x)\big)P(\d x)
=
\int K(x) P(\d x)
\\&=
\frac{1}{n!}\sum_{\sigma\in\S_n}\int\delta_{x_\sigma} P(\d x)
=
\frac{1}{n!}\sum_{\sigma\in\S_n}(r_\sigma)_*P
=
P,
\end{align*}
where $r_\sigma\colon M^n\ni x \mapsto x_\sigma \in M^n$ is the action of the permutation $\sigma$ on the configuration space, and where the last equality follows from the symmetry of $P$.
\end{proof}

Hewitt and Savage \cite[Section~12]{hewitt_symmetric_1955} characterized the set of extremal points within the convex set of symmetric probability distributions on $M^n$, for short, \emph{extremal distributions}.
Moreover, they proved that every symmetric probability distribution is a mixture of extremal distributions and called such mixtures \emph{presentable}. 
As a corollary to Lemma~\nmb!{6.1}, one obtains an elementary proof of these facts.
The more widely studied case of infinite configurations is discussed in \nmb!{6.3} and \nmb!{6.4}.

\begin{corollary}[Finite Hewitt--Savage theorem]
\nmb.{6.2}
The extremal points in the convex set $\probspace(M^n)_{\S_n}$ of symmetric distributions are exactly of the form 
$\frac{1}{n!}\sum_{\sigma\in\S_n}\delta_{x_\sigma}$, $x \in M^n$. 
Moreover, all symmetric probability distributions on $M^n$ are presentable.
\end{corollary}

\begin{proof}
The map (\nmb!{6.1}.2) is a linear bijection and therefore maps extremal points in its domain to extremal points in its range. 
The extremal points in the domain are easily identified as the Dirac measures. 
The image of a Dirac measure $\delta_{\bar x}$ with $\bar x =\proj(x)\in M^n/\S_n$ is the distribution $\frac{1}{n!}\sum_{\sigma\in\S_n}\delta_{x_\sigma}$.
The range of the map (\nmb!{6.1}.2) consists of mixtures of such distributions, i.e., presentable distributions. 
Moreover, as (\nmb!{6.1}.2) is surjective, all symmetric distributions are presentable.
\end{proof}

The following lemma characterizes distributions of infinite samples, thereby generalizing the corresponding result \nmb!{6.1} for finite samples. 
The full permutation group $S_{\mathbb N}$ of the natural numbers is too large for our purpose.
Instead, we consider the \emph{infinite permutation group} $\S_{(\mathbb N)}\coloneqq \bigcup_{n\in\mathbb N}\S_n$, which acts upon the \emph{infinite configuration space} $M^{\mathbb N}\coloneqq \prod_{n\in\mathbb N}M$.
A probability distribution on $M^{\mathbb N}$ is called \emph{symmetric} if it is $\S_{(\mathbb N)}$-invariant, and the set of symmetric distributions is denoted by $\probspace(M^{\mathbb N})_{\S_{(\mathbb N)}}$. 
The correct space of \emph{infinite samples}, which leads to a generalization of \nmb!{6.1}, is not the quotient space $M^{\mathbb N}/\S_{(\mathbb N)}$, but the space $\probspace(M)$. 
This is demonstrated in Example~\nmb!{6.5} and is in line with the limiting result \nmb!{4.8}.
\begin{lemma}[Distributions of infinite samples]
\nmb.{6.3}
Probability distributions on the infinite sample space $\probspace(M)$ correspond exactly to symmetric probability distributions on the configuration space $M^{\mathbb N}$. 
\end{lemma}

\begin{proof}
For some fixed point $o \in M$, define a projection from infinite configuration space to infinite sample space as follows:
\begin{align*}
\proj\colon M^{\mathbb N}\to \probspace(M), 
&&
\proj(x)\coloneqq
\left\{
\begin{aligned}
&\lim_{n\to\infty}\frac1n\sum_{i=1}^n \delta_{x_i},
&&\text{if the weak limit exists,}
\\
&\delta_o, 
&&\text{otherwise,}
\end{aligned}
\right.
\end{align*}
The push-forward along this projection restricts to the following map from symmetric distributions to probability distributions on infinite sample space $\probspace(M)$:
\begin{align*}
\proj_*\colon \probspace(M^{\mathbb N})_{\S_{(\mathbb N)}} \to \probspace(\probspace(M)).
\end{align*}
We claim that the map $\proj_*$ is an inverse of the map
\begin{align*}
\probspace(\probspace(M)) \ni Q \mapsto \int_{\probspace(M)}P^{\mathbb N}\ Q(\d P) \in \probspace(M^{\mathbb N})_{\S_{(\mathbb N)}},
\end{align*}
where $P^{\mathbb N}\coloneqq \bigotimes_{n\in\mathbb N}P$ denotes the product distribution on $M^{\mathbb N}$, and where the integral is a measure-valued Pettis integral.
Note that the distributions on the right-hand side are laws of conditionally \iid{} sequences of $M$-valued random variables.
To prove the claim, we appeal to the infinite-sample version \nmb!{6.4} of the Hewitt--Savage theorem, which states that symmetric distributions coincide exactly with presentable distributions, i.e., with Pettis integrals as above.
For any $P \in \probspace(M)$, the weak law of large numbers implies $\proj(x)=P$ for $P^{\mathbb N}$-almost every $x \in M^{\mathbb N}$.
This implies $\proj_*(P^{\mathbb N})=\delta_P$.
Consequently, every $Q \in \probspace(\probspace(M))$ satisfies
\begin{align*}
\proj_*\left(\int_{\probspace(M)}P^{\mathbb N}\ Q(\d P)\right)
=
\int_{\probspace(M)}\proj_*(P^{\mathbb N})\ Q(\d P)
=
\int_{\probspace(M)}\delta_P\ Q(\d P)
=
Q.
\end{align*} 
This proves the claim and establishes the desired one-to-one correspondence.
\end{proof}

The above proof uses the well-known Hewitt--Savage theorem \cite{hewitt_symmetric_1955}, which is a generalization of Corollary~\nmb!{6.2} to infinite sample spaces.
As before, \emph{presentable distributions} are defined as mixtures of \emph{extremal distributions}, i.e., of extremal points in the convex set of symmetric distributions. 

\begin{theorem}[Infinite Hewitt--Savage theorem {\cite{hewitt_symmetric_1955}}]
\nmb.{6.4}
The extremal points in the convex set $\probspace(M^{\mathbb N})_{\S_{(\mathbb N)}}$ of symmetric distributions are exactly the product distributions $P^{\mathbb N}\coloneqq\bigotimes_{n\in\mathbb N} P$ with $P \in \probspace(M)$.
Moreover, all symmetric distributions on $M^{\mathbb N}$ are presentable. 
\end{theorem}

This result is asymptotically consistent with its finite-sample counterpart \nmb!{6.2}.
Indeed, by the Diaconis--Freedman theorem \cite{diaconis_finite_1980} symmetric distributions on $M^n$ are close to mixtures of product distributions for large $n$.
More precisely, the total variation distance from $k$-dimensional marginal distributions of elements of $\probspace(M^n)_{\S_n}$ to mixtures of product distributions is at most $k(k-1)/n$.

The following example shows that the correspondence \nmb!{6.3} between probability distributions on sample space and symmetric probability distributions on configuration space fails if the sample space is defined as $M^{\mathbb N}/\S_{(\mathbb N)}$ instead of $\probspace(M)$. 

\begin{example}[Infinite sample space]
\nmb.{6.5}
There is a probability distribution on $M^{\mathbb N}/\S_{(\mathbb N)}$ which does not correspond to any symmetric probability distribution on $M^{\mathbb N}$.
\end{example}

\begin{proof}
In analogy to \nmb!{6.3} we say that a probability distribution $Q$ on $M^{\mathbb N}/\S_{(\mathbb N)}$ corresponds to a symmetric probability distribution $P$ on $M^{\mathbb N}$ if $Q=\proj_*P$, where $\proj\colon M^{\mathbb N}\to M^{\mathbb N}/\S_{(\mathbb N)}$ is the canonical projection. 
For any such $P$, the weak limit
$\lim_{n\to\infty}\frac1n\sum_{i=1}^n \delta_{x_i}$
exists for $P$-almost every $x \in M^{\mathbb N}$, as shown in the proof of \nmb!{6.3}.
Moreover, this limit is invariant under the action of $\S_{(\mathbb N)}$ on $M^{\mathbb N}$ because every permutation in $\S_{(\mathbb N)}$ affects only finitely many indices. 
Thus, if $Q$ corresponds to some $P$, then the limit $\lim_{n\to\infty}\frac1n\sum_{i=1}^n \delta_{\bar x_i}$ is well-defined and exists for $Q$-almost every $\bar x \in M^{\mathbb N}/\S_{(\mathbb N)}$.
However, it is easy to construct a distribution $Q$ on $M^{\mathbb N}/\S_{(\mathbb N)}$ which does not have this property. 
Indeed, assuming that $M$ contains at least two points, one may construct a sequence of points $x_i \in M$ such that $\frac1n\sum_{i=1}^n\delta_{x_i}$ does not converge weakly as $n\to\infty$. 
Then $Q\coloneqq \delta_{\bar x}$ with $\bar x\coloneqq\proj(x)$ is the desired counter-example.
\end{proof}

We next investigate random samples and random configurations.
For this purpose, we fix a probability space $(\Omega,\mathcal F, \mathbb P)$ on which all random variables are defined.
A \emph{random configuration} is a random variable in $M^n$ or $M^{\mathbb N}$, depending on the finite versus infinite case. 
Similarly, a \emph{random sample} is a random variable in $M^n/\S_n$ or $\probspace(M)$, respectively.
A random configuration is called \emph{exchangeable} if its law is symmetric, i.e., invariant under permutations in $\S_n$ or $\S_{(\mathbb N)}$, respectively.
The following characterization is analogous to \nmb!{6.1}--\nmb!{6.4}.

\begin{corollary}[Random configurations and samples]
\nmb.{6.6} 
Random samples correspond exactly (possibly after passing to an extended probability space) to exchangeable configurations, which in turn correspond exactly to conditionally \iid{} $M$-valued random variables.
This statement applies to finite and infinite configurations and samples, respectively.
\end{corollary}

\begin{proof}
This can be shown in analogy to \nmb!{6.1}--\nmb!{6.4}, working with random variables instead of their laws.
The extension of the probability space is necessary, unless the given probability space is already sufficiently rich, for implementing the random ordering in the proof of \nmb!{6.1} and the \iid{} sampling in the proof of \nmb!{6.3}.
\end{proof}

\section{Asymptotic properties of polymeans}
\nmb0{7}
\label{sec:asymptotics}
Polymeans, similar to Fr\'echet means \cite{huckemann_intrinsic_2011}, satisfy a law of large numbers and a central limit theorem under suitable conditions, as shown next.
We refer to \nmb!{5.1}, \nmb!{5.8}, and \nmb!{5.9} for their definition.
Throughout this section, 
$(M,d)$ is a separable complete connected path-metric space, 
$p \in [1,\infty)$, 
and $q\in\mathbb N_{>0}$.
The space $M$, as well topological products and quotients thereof, are endowed with the corresponding Borel sigma algebras. 
For some probability distribution $P\in\probspace^p(M)$, we consider a sequence of independent $P$-distributed random variables $(x_i)_{i\in\mathbb N}$ defined on a complete probability space $(\Omega,\mathcal F, \mathbb P)$. 
The corresponding $n$-samples are denoted by $\bar x_n\coloneqq\proj(x_1,\dots,x_n) \in M^n/\S_n$.
We write $\mu_n\subset(M^n/\S_n)_q$ for the set of $q$-means of $\bar x_n$, 
$\bar y_n\in(M^n/\S_n)_q$ for a measurable selection of $q$-means of $\bar x_n$, 
and $\bar z_n\in (M^q/\S_q)_{\text{reg}}$ for a measurable selection of unweighted $q$-means of $\bar x_n$.
It will be convenient to identify the samples $\bar x_n, \bar y_n,\bar z_n$ with their empirical laws $P_n,Q_n,R_n$, respectively, using the isometry \nmb!{4.7} between $M^n/\S_n$ and $\probspace_n(M)$. 
The population counterparts of the above empirical objects are denoted by $\mu_0$, $\bar y_0$, $\bar z_0$, $Q_0$, and $R_0$, respectively.
Note that all of these objects belong to one and the same path-metric space $\probspace^p(M)$ thanks to the isometric embedding \nmb!{4.8} of finite into infinite sample spaces.

\begin{definition}[Strong consistency {\cite{ziezold_expected_1977}}]
\nmb.{7.1}
The empirical $q$-means $\mu_n$ are called \emph{strongly consistent} estimators for the set $\mu_0$ of population $q$-means if 
\begin{equation*}
\mathbb P\left[\bigcap_{n=1}^\infty \overline{\bigcup_{k=n}^\infty\mu_k}\subseteq \mu_0\right]=1.
\end{equation*}
\end{definition}

Note that strong consistency is equivalent to the following statement: with probability 1, any accumulation point of the sets $\mu_n$ belongs to $\mu_0$.

\begin{lemma}[Strong consistency]
\nmb.{7.2}
The empirical $q$-means $\mu_n$ are strongly consistent estimators for the population $q$-means $\mu_0$.
\end{lemma}

This statement is a consequence of the Gamma-convergence of the functionals which are minimized by $\mu_n$ and $\mu_0$, respectively, as shown in the following proof. 
A similar argument is used in \cite{ziezold_expected_1977} and \cite[Theorem~A.3]{huckemann_intrinsic_2011}. These proofs are longer because implications of Gamma-convergence are re-proven there.

\begin{proof}[Proof]
The empirical $q$-means $\mu_n$ are the minimizers of the functional
\begin{align*}
F_n\colon\probspace^p(M)_q&\to\mathbb R_+,  
&
F_n(Q)=
\left\{\begin{aligned}
&\bar d_p(P_n,Q),&&Q \in (M^n/\S_n)_q,
\\
&\infty,&&Q \notin (M^n/\S_n)_q.
\end{aligned}\right.
\end{align*}
Similarly, the population $q$-means $\mu$ are the minimizers of the functional
\begin{align*}
F\colon\probspace^p(M)_q&\to\mathbb R_+,  
&
F(Q)=\bar d_p(P,Q).
\end{align*}
The empirical laws $P_n$ converge to the population law $P$ in the Wasserstein metric $\bar d_p$ by \cite[Proposition~2.2.6]{panaretos2020invitation}.
We claim that this implies Gamma-convergence $F_n\to F$. 
To prove the claim, note that for any converging sequence $Q_n\to Q$ in $\probspace^p(M)_q$, 
\begin{equation*}
F(Q)
=
\bar d_p(P,Q)
=
\lim_{n\to\infty} \bar d_p(P_n,Q_n)
\leq 
\liminf_{n\to\infty} F_n(Q_n).
\end{equation*}
Moreover, any $Q\in\probspace^p(M)_q$ can be approximated in the $\bar d_p$-distance by a sequence $Q_n\in(M^n/\S_n)_q$. 
Indeed, $Q$ is of the form $Q=\sum_{i=1}^q w_i\delta_{x_i}$ for some $x_i \in M$ and $w_i\in[0,1]$, and the approximations $Q_n$ may be defined by rounding the weights to the nearest multiples of $1/n$.
For any such approximating sequence $Q_n\to Q$ one has
\begin{equation*}
F(Q)
=
\bar d_p(P,Q)
=
\lim_{n\to\infty} \bar d_p(P_n,Q_n)
= 
\lim_{n\to\infty} F_n(Q_n).
\end{equation*}
This proves that $F_n$ Gamma-converges to $F$. 
Thus, the accumulation points of $F_n$-minimizers are $F$-minimizers, which is exactly strong consistency. 
\end{proof}

If the empirical $q$-means are strongly consistent and the population $q$-mean is unique, then any measurable selection $Q_n$ of empirical $q$-means converges in probability to the population $q$-mean $Q_0$. 
In this situation one may inquire about the rate of convergence $Q_n\to Q_0$.
As an auxiliary first step, the following lemma shows that $Q_n$ possesses the same best-approximation property as $Q_0$, up to some error terms.
Controlling these error terms leads to the convergence rate established subsequently in \nmb!{7.4}.

\begin{lemma}[Error bound]
\nmb.{7.3}
Assume that $P \in \probspace^{2p}(M)$, 
let $Q_0 \in \probspace(M)_q$ be a $q$-mean of $P$, 
assume that $Q_0$ is distinct from $P$,
and for each $n \in \mathbb N$, let $Q_n\in \probspace_n(M)_q$ be a $q$-mean of the empirical law $P_n$.
Then
\begin{align*}
\bar d_p(P,Q_n) - \bar d_p(P,Q_0) 
\leq
\bar d_p(P_n,P) + O_{\mathbb P}(n^{-1/2}).
\end{align*} 
\end{lemma}

\begin{proof}
Let $K\colon M\to \probspace(M)$ be an optimal transport map from $P$ to $Q_0$, i.e., 
\begin{align*}
Q_0 &= \int_M K(x) P(\d x),
&
\bar d_p(P,Q_0)
&=
\left(\int_M\int_M d(x,y)^p K(x,\d y)P(\d x)\right)^{1/p}.
\end{align*}
Such a transport map can be obtained from an optimal coupling between $P$ and $Q_0$ via disintegration.
Then $K$ is also a transport map between $P_n$ and $\tilde Q_n$, where
\begin{equation*}
\tilde Q_n
\coloneqq
\int_M K(x) P_n(\d x) \in \probspace_n(M)_q. 
\end{equation*}
By the triangle inequality and the best-approximation property of the polymeans, 
\begin{align*}
\bar d_p(P,Q_n)
&\leq 
\bar d_p(P_n,Q_n)+\bar d_p(P_n,P)
\leq
\bar d_p(P_n,\tilde Q_n)+\bar d_p(P_n,P)
\\&\leq 
\left(\int_M\int_M d(x,y)^p K(x,\d y)P_n(\d x)\right)^{1/p}+\bar d_p(P_n,P).
\end{align*}
Rewriting the right-hand side using the defining properties of $K$ leads to the estimate
\begin{equation*}
\bar d_p(P,Q_n)
\leq 
\left(\bar d_p(P,Q_0)^p + \int_M\int_M d(x,y)^p K(x,\d y)(P_n-P)(\d x)\right)^{1/p}+\bar d_p(P_n,P).
\end{equation*}
By the central limit theorem, the random variables
\begin{equation*}
n^{1/2}\int_M\int_M d(x,y)^p K(x,\d y) (P_n-P)(\d x)
\end{equation*}
converge in distribution to a normal random variable. 
As $d_p(P,Q_0)>0$, this establishes the lemma.
The central limit theorem may be applied thanks to the square-integrability condition
\begin{align*}
\int_M\left(\int_M d(x,y)^p K(x,\d y)\right)^2 P(\d x)
&\leq 
\int_M\int_M d(x,y)^{2p} K(x,\d y) P(\d x)
\\&\leq 
\sum_{y \in \operatorname{supp}(Q_0)}\int_M d(x,y)^{2p} P(\d x) 
< 
\infty.
\qedhere
\end{align*} 
\end{proof}

The bound in \nmb!{7.3} involves the Wasserstein distance $\bar d_p(P_n,P)$ between a distribution $P$ and the empirical distribution $P_n$ of an $n$-sample, which is itself a random variable.  
On $M=\mathbb R^d$ it has been shown for distributions $P$ with sufficiently many moments that  $\|\bar d_p(P_n,P)\|_{L^p(\Om)}$ 
is of the order $n^{-1/\max\{d,2p\}}$, with an additional logarithmic factor if $d=2p$ \cite[Theorem~1]{fournier2015rate}.
This paper also gives references for improved rates under more stringent conditions on $P$.
The case of non-flat $M$ is largely open. 

Using Lemma~\nmb!{7.3}, the following theorem bounds the rate at which the empirical $q$-means $Q_n$ converge to the population $q$-mean $Q_0$. 
Besides the distance $\bar d_p(P_n,P)$, it also involves a real number $\alpha$, which quantifies the coercivity of the Wasserstein distance $\bar d_p(P,\cdot)$ near a minimizer $Q_0$ in the $q$-skeleton and depends on the subspace geometry of the $q$-skeleton within Wasserstein space. 

\begin{theorem}[Convergence rate]
\nmb.{7.4}
Let $P \in \probspace^{2p}(M)$, 
let $Q_n\in \probspace_n(M)_q$ be a sequence of $q$-means of $P_n$ converging in probability to a population $q$-mean  $Q_0\in\probspace(M)_q$,
and assume for some $\alpha>0$ and $c>0$ that
\begin{equation*}
\bar d_p(P,Q)-\bar d_p(P,Q_0)\geq c \bar d_p(Q,Q_0)^\alpha 
\end{equation*}
for all $Q \in (M^n/\S_n)_q$ near $Q_0$.
Then
\begin{align*}
\bar d_p(Q_n,Q_0)
=
O_{\mathbb P}(\bar d_p(P_n,P)^{1/\alpha})
+
O_{\mathbb P}(n^{-1/(2\alpha p)}).
\end{align*}
\end{theorem}

\begin{proof}
The error bound \nmb!{7.3} together with the assumption on the distance function imply that 
\begin{align*}
c\bar d_p(Q_n,Q_0)^\al 
&\leq
\bar d_p(P,Q_n)-\bar d_p(P,Q_0)
\leq
\bar d_p(P_n,P) + O_{\mathbb P}(n^{-1/(2p)}).
\end{align*}
Taking the $\al$-th root establishes the theorem.
\end{proof}

It remains open if weighted $q$-means are asymptotically normal after a suitable rescaling. 
However, we will answer this question affirmatively for unweighted $q$-means, defined in \nmb!{5.8}. 
Note that these are strongly consistent thanks to the strong consistency \nmb!{7.2} of weighted $q$-means. 

\begin{definition}[Asymptotic normality]
\nmb.{7.5}
Assume that $M$, and consequently also the regular stratum $(M^q/\S_q)_{\text{reg}}=\probspace_q(M)_{\text{reg}}$, is a manifold.
Fix a regular sample $R_0 \in \probspace_q(M)_{\text{reg}}$ and a symmetric bilinear form $\Sigma$ on the tangent space at $R_0$.
Then a sequence $R_1,R_2,\dots$ of random elements in $\probspace_q(M)$ is called \emph{asymptotically normal} with mean $R_0$ and covariance $\Sigma$ if for some (equivalently, every) coordinate chart $(U,u)$ around $R_0$, the sequence $\sqrt{n}\mathbbm{1}_U(R_n) u(R_n)$ converges in law to a normal distribution $\mathcal N(0,u_*(\Sigma))$.
\end{definition}

The chart independence in this definition is a consequence of the delta method \cite[Theorem~3.1]{vandervaart2000asymptotic}.
We then get the following asymptotic result.

\begin{theorem}[Asymptotic normality]
\nmb.{7.6}
Let $M$ be a manifold with Riemannian path metric $d$ and assume that conditions \thetag{1}--\thetag{2} in the proof below hold true.
Then any sequence $R_1,R_2,\dots$ of unweighted $q$-means of $P_1,P_2,\dots$, which converges in probability to a unique unweighted population $q$-mean $R_0$, is asymptotically normal. 
\end{theorem}

\begin{proof}
As before, we use \nmb!{4.7} to identify the unweighted $q$-means $R_0,R_1,R_2,\ldots \in \probspace_q(M)_{\text{reg}}$ with the corresponding $q$-samples $\bar z_0,\bar z_1,\bar z_2,\ldots \in (M^q/\S_q)_{\text{reg}}$.
In the notation of \cite{eltzner_smeary_2019} and \cite{huckemann_intrinsic_2011}, and in line with Definition~\nmb!{5.8} of unweighted $q$-means, we define the Fr\'echet functional 
\begin{equation*}
\bar{\rho}\colon
M\times (M^q/\S_q)_{\text{reg}}
\ni (x,\bar z)
=
(x,\proj(z))
\mapsto
\min_{i\in\{1,\dots,q\}} d(x,z_i)^p.
\end{equation*}
Then the unweighted $q$-means $\bar z_n$ minimize the functional
\begin{equation*}
P_n\bar{\rho}\colon
(M^q/\S_q)_{\text{reg}}
\ni \bar z
\mapsto
\int_M\bar\rho(x,\bar z)P_n(\d x),
\end{equation*}
and the unweighted population $q$-mean $\bar z_0$ minimizes the functional
\begin{equation*}
P\bar{\rho}\colon
(M^q/\S_q)_{\text{reg}}
\ni \bar z
\mapsto
\int_M\bar\rho(x,\bar z)P(\d x).
\end{equation*}
To verify the conditions of \cite{eltzner_smeary_2019} we make the following assumptions: \\

\begin{enumerate}
\item 
The following sets have zero probability under $P$:
\begin{gather*}
\{\bar z_{0,1},\dots,\bar z_{0,q}\}, 
\quad 
\Cut(\bar z_{0,1})\cup\dots\cup\Cut(\bar z_{0,q}),
\\
\{x\in M: \exists i\neq j \in \{1,\dots,q\}:  d(x,\bar z_{0,i})=d(x,\bar z_{0,j})=\rho(x,\bar z_0)\}.
\end{gather*} 
\item
The function $P\bar{\rho}$ defined above has a non-degenerate Hessian at $\bar z_0$.

\end{enumerate}

Note that the first assumption guarantees for $P$-almost every $x \in M$ the existence of the Riemannian gradient of the function $\bar\rho(x,\cdot)$ at $\bar z_0$. 
Indeed, the only points $x$ where the gradient may fail to exist are the points $\bar z_{0,i}$, their cut loci $\Cut(\bar z_{0,i})$, and the locations which are closest to more than one $\bar z_{0,i}$.
A further condition of \cite{eltzner_smeary_2019} to be verified for all $x \in M$ is that the function $\bar{\rho}(x,\cdot)$ is uniformly continuous on bounded domains with respect to the metric $\bar d_p$ on $M^q/\S_q$. 
This follows from the estimate
\begin{align*}
|\bar{\rho}(x,\proj(z'))-\bar{\rho}(x,\proj(z))|
\leq 
\max_{i\in\{1,\dots,q\}} \min_{j\in\{1,\dots,q\}} d(z_i,z_j')^p
\leq 
q \bar d_p(\proj(z),\proj(z'))^p.
\end{align*}
Thus, we have verified the conditions of \cite[Theorem~11]{eltzner_smeary_2019}, 
and it follows that the sequence $\bar z_n$ or equivalently $R_n$ is asymptotically normal.
\end{proof}

The asymptotic normality of unweighted $q$-means generalizes from independent to exchangeable observations $x_1,x_2,\dots$ under certain conditions. 
Equivalently, as shown in \nmb!{6.6}, the observations can be seen as random elements in an infinite sample space.

\begin{corollary}[Asymptotic normality, exchangeable observations]
\nmb.{7.7}
Theorem~\nmb!{7.6} extends to exchangeable sequences of (not necessarily independent) observations $x_1,x_2,\dots$, provided that condition~\thetag{1} in the proof below is satisfied. 
\end{corollary}

\begin{proof}
By the infinite Hewitt–Savage theorem \nmb!{6.4} and its Corollary~\nmb!{6.6}, the exchangeable sequence $x_1,x_2,\dots$ is \iid{}  conditionally on some sigma algebra $\mathcal G$. 
It follows from \nmb!{7.6} that conditionally on $\mathcal G$, the sequence $R_1,R_2,\dots$ is asymptotically normal with mean $R_0$ and covariance $\Sigma$, for some $\mathcal G$-measurable symmetric bilinear form $\Sigma$ on the tangent space of $(M^q/\S_q)_{\text{reg}}$ at $R_0$. 
The covariance $\Sigma$ can be computed explicitly as follows. 
Let $\bar \rho$ be defined as in the proof of \nmb!{7.6}, 
and recall that the gradient of the function $\bar\rho(x,\cdot)$ evaluated at $R_0$ exists for $P$-almost every $x \in M$.
Therefore, for any $i\in\mathbb N_{>0}$, one may define the random variable $X_i$ as the gradient of the random function $\bar\rho(x_i,\cdot)$ evaluated at $R_0$.
Accordingly, $X_i$ is a random variable with values in the tangent space of $(M^q/\S_q)_{\text{reg}}$ at $R_0$.
Let $\bar H$ denote the Hessian of the function $P\bar\rho$ at $R_0$. 
Thanks to the non-degeneracy assumption in \nmb!{7.6}, $\bar H$ is an automorphism on the tangent space of $(M^q/\S_q)_{\text{reg}}$ at $R_0$, and we denote its inverse by $\bar H^{-1}$.
Then the covariance $\Sigma$ is given by  \cite[Theorem~11]{eltzner_smeary_2019}
\begin{equation*}
\Sigma = \frac14\Cov[\bar H^{-1}(X_1)\otimes \bar H^{-1}(X_1)|\mathcal G].
\end{equation*}
To ensure that $\Sigma$ is deterministic, we make the following assumption:
\begin{equation*}
\tag{1}
\mathbb E[X_1]=0,
\qquad
\Cov(X_1,X_2)=0,
\qquad
\Cov(X_1\otimes X_1,X_2\otimes X_2)=0.
\end{equation*}
Define $B=\mathbb E[X_1\otimes X_1]$ and $C=\Cov[X_1\otimes X_1]$.
Then the relations
\begin{align*}
0
&=
\mathbb E[X_1\otimes X_2]
=
\mathbb E[\mathbb E[X_1\otimes X_2|\mathcal G]]
=
\mathbb E[\mathbb E[X_1|\mathcal G]^{\otimes 2}],
\\
C
&=
\mathbb E[(X_1\otimes X_1-B)\otimes(X_2\otimes X_2-B)]
=
\mathbb E[\mathbb E[X_1\otimes X_1-B|\mathcal G]^{\otimes2}],
\end{align*}
show that \thetag{1} is equivalent to 
\begin{equation*}
\mathbb E[X_1|\mathcal G]=0,
\qquad
\mathbb E[X_1\otimes X_1|\mathcal G]=B.
\end{equation*}
Therefore, $\Sigma=(\bar H^{-1}\otimes \bar H^{-1})(B)$ is deterministic, as claimed.
As $R_0$ and $\Sigma$ are deterministic, the sequence $R_1,R_2,\dots$ is not only conditionally but also unconditionally asymptotically normal.
See \cite[Theorem 9.2.1]{chow_probability_1997} for further details in the Euclidean case. 
\end{proof}

\begin{appendix} 
\section{}
\nmb0{8}

\begin{definition}[Path metrics {\cite{gromov1999metric, burago2001course}}]
\nmb.{8.1}
In any metric space $(M,d)$, for any real numbers $s\leq t$, the \emph{length} of a continuous curve $c\colon [s,t]\to M$ is defined as
\begin{equation*}
\ell(c) 
= 
\sup_{\substack{n\in\mathbb N\\s=u_0\leq\dots\leq u_n=t}} \sum_{i=0}^{n-1} d\big(c(u_i),c(u_{i+1})\big) \in [0,\infty].
\end{equation*} 
The curve is said to have \emph{constant speed} $v \in \mathbb R_{\geq 0}$ if 
$\ell(c|_{[u_1,v_1]}) = v|u_1-u_2|$ for all $s\le u_1<u_2\le t$. 
The metric space is called a \emph{path-metric space} if the distance between any pair of points equals the infimum of the lengths of continuous curves joining the points. 
A \emph{minimizing geodesic} in a metric space $(M,d)$ is a continuous curve whose length equals the distance between its end points. A \emph{geodesic} is a curve whose restriction to any sufficiently small subinterval is a minimizing geodesic. 
\end{definition}

\begin{theorem}[Hopf--Rinov theorem {\cite[1.9]{gromov1999metric}}] 
\nmb.{8.2}
If $(M,d)$ is a connected complete locally compact path-metric space then: 
\begin{enumerate}
\item Closed balls are compact, or, equivalently, each bounded closed subset of $M$ is compact.
\item Any two points can be joined by a minimizing geodesic.
\end{enumerate}
\end{theorem}

\begin{theorem}[Characterization of path metrics {\cite[Theorem 1.8]{gromov1999metric}}]
\nmb.{8.3}
The following properties of a metric space $(M,d)$ are equivalent:
\begin{enumerate}
\item For any points $x,y\in M$ and $r>1/2$ there exists a point $z\in M$ such that 
$$
\max\{d(x,z),d(z,y)\} \le r d(x,y).
$$
\item For all $x,y\in M$ and $r_1,r_2>0$ with $r_1+r_2\le d(x,y)$ we have 
\begin{align*}
d(B(x,r_1),B(y,r_2))&:= \inf\{d(x',y'): d(x',x)\le r_1, d(y',y)\le r_2\}
\\&
\le d(x,y) -r_1 -r_2.
\end{align*}
\end{enumerate}
Every path-metric space has these properties. Conversely, a complete metric space with property \thetag{1} or \thetag{2} is a path-metric space.
\end{theorem}

\begin{definition}[Lagrangian actions]
\nmb.{8.4}
Following \cite[Definition~7.11]{villani2008optimal}, 
a \emph{Lagrangian energy--action pair} $(E,A)$ on a topological space $M$ is a family of energy functionals $E^{s,t}\colon M\times M\to \mathbb R$ and action functionals $A^{s,t}\colon C([s,t],M)\to\mathbb R$, indexed by real numbers $s\leq t$, which satisfies the following three properties:
\begin{enumerate}
\item for all $r\leq s\leq t$, $A^{r,s}+A^{s,t}=A^{r,t}$,
\item for all $s\leq t$ and $x,y \in M$, 
\begin{equation*}
E^{s,t}(x,y) 
=
\inf_{\substack{c \in C([s,t],M)\\c(s)=x, c(t)=y}} A^{s,t}(c).
\end{equation*}
\item for all $s\leq t$ and $c \in C([s,t],M)$, 
\begin{equation*}
A^{s,t}(c)
=
\sup_{\substack{n\in\mathbb N\\s=u_0\leq\dots\leq u_n=t}} \sum_{i=0}^{n-1} E^{u_i,u_{i+1}}(c(u_i),c(u_{i+1})).
\end{equation*}
\end{enumerate}

Curves which assume the minimum in (2) are called \emph{minimizing curves} for $(E,A)$. 
\end{definition}

Examples of Lagrangian energy--action pairs on path-metric spaces $(M,d)$ are $(d,\ell)$ as well as the functionals described in the following lemma, which are related to the Riemannian or Finsler energy.

\begin{lemma}[Lagrangian actions]
\nmb.{8.5}
For any path-metric space $(M,d)$ and $p \in (1,\infty)$, 
the following defines a Lagrangian energy--action pair $(E,A)$: 
\begin{align*}
E^{s,t}(x,y)
&=
\frac{d(x,y)^p}{|s-t|^{p-1}},
&
A^{s,t}(c) 
&= 
\sup_{\substack{n\in\mathbb N\\s=u_0\leq \dots\leq u_n=t}}
\sum_{i=0}^{n-1}
\frac{d(c(u_i),c(u_{i+1}))^p}{|u_i-u_{i+1}|^{p-1}}.
\end{align*}
Minimizing curves for $(E,A)$ are exactly constant-speed minimizing geodesics. 
\end{lemma}

\begin{proof}
Properties (1) and (3) of Lagrangian actions hold by definition.
Property (2) can be verified as follows: as $(M,d)$ is a path-metric space, the definition of the energy implies for any real numbers $s\leq t$ and points $x,y \in M$ that
\begin{align*}
E^{s,t}(x,y)
&=
\inf_{\substack{c \in C([s,t],M)\\c(s)=x, c(t)=y}}
\sup_{\substack{n\in\mathbb N\\s=u_0\leq\dots\leq u_n=t}} 
\frac{\left(\sum_{i=0}^{n-1} d(c(u_i),c(u_{i+1}))\right)^p}{|s-t|^{p-1}}.
\end{align*}
Estimating the right-hand side using H\"older's inequality yields
\begin{align*}
E^{s,t}(x,y)
&\leq
\inf_{\substack{c \in C([s,t],M)\\c(s)=x, c(t)=y}}
\sup_{\substack{n\in\mathbb N\\s=u_0\leq\dots\leq u_n=t}} 
\sum_{i=0}^{n-1} \frac{d(c(u_i),c(u_{i+1}))^p}{|s-t|^{p-1}}
=
\inf_{\substack{c \in C([s,t],M)\\c(s)=x, c(t)=y}}
A^{s,t}(c).
\end{align*}
For constant-speed curves, H\"older's inequality is an equality. 
Moreover, any continuous curve can be reparameterized to constant speed. 
Therefore, the preceding inequality is actually an equality. 
This shows (2). 

The statement about minimizing curves hinges on the  following H\"older inequality: for all $u\leq v\leq w$ in the domain of a continuous curve $c\colon[s,t]\to M$,
\begin{align*}
\hspace{2em}&\hspace{-2em}
d(c(u),c(v))+d(c(v),c(w)
\\&=
\frac{d(c(u),c(v))}{|u-v|^{(p-1)/p}}|u-v|^{(p-1)/p}+\frac{d(c(v),c(w)}{|v-w|^{(p-1)/p}}|v-w|^{(p-1)/p}
\\&\leq
\big(E^{u,v}(c(u),c(v))+E^{v,w}(c(v),c(w) \big)^{1/p}
|u-w|,
\end{align*}
with equality if and only if the vector $\big(d(c(u),c(v)),d(c(v),c(w))\big)$ in $\mathbb R^2$ is parallel to the vector $(v-u,w-v)$.

Let $c\colon[s,t]\to M$ be a continuous curve. 
Then $c$ is a minimizing geodesic with constant speed if and only if it satisfies for all $u\leq v\leq w$ in $[s,t]$ that 
\begin{equation*}
\begin{aligned}
d(c(u),c(v))+d(c(v),c(w) 
&= 
d(c(u),c(w)),
\\
d(c(u),c(v))&=d(c(s),c(t))|u-v|.
\end{aligned}
\end{equation*}
Equivalently, by the above H\"older inequality, it holds for all $u\leq v\leq w$ in $[s,t]$ that
\begin{equation*}
E^{u,v}(c(u),c(v))+E^{v,w}(c(v),c(w) 
= 
E^{u,w}(c(u),c(w)),
\end{equation*}
which means that $c$ minimizes the energy--action pair $(E,A)$. 
\end{proof}

\begin{lemma}[Atomic distributions]
\nmb.{8.6}
  Let $M$ be a metric space or, more generally, a first-countable space. 
Then the set $\probspace_n(M)$ coincides with the set of $\{0,1/n,\dots,1\}$-valued probability distributions on $M$.
\end{lemma}

\begin{proof}
Clearly, every distribution in $\probspace_n(M)$ takes values in $\{0,1/n,\dots,1\}$.
Conversely, assume that $P$ is a $\{0,1/n,\dots,1\}$-valued probability distribution. 
Let $x \in M$, and let $(U_i)_{i\in\mathbb N}$ be a decreasing basis of open neighborhoods of $x$. 
If $\min_{i\in\mathbb N}P(U_i)$ vanishes, then it vanishes for sufficiently large $i$, and consequently $x$ does not belong to the support of $P$. 
Otherwise, $P(\{x\})=\min_{i\in\mathbb N}P(U_i)\geq \frac1n$, which can be the case for only finitely many $x \in M$. 
Therefore, the support of $P$ is a finite set.
It follows that $P$ is a weighted sum of Dirac measures at distinct points in $M$.
Necessarily, the weights are multiples of $1/n$.  
\end{proof}
\end{appendix}

\section*{Acknowledgements}
The authors would like to thank Fran\c{c}ois-Xavier Vialard for helpful discussions. 
P. Harms was funded by the National Research Foundation Singapore under the award NRF-NRFF13-2021-0012 and by Nanyang Technological University Singapore under the award NAP-SUG.
X. Pennec was funded by the ERC grant Nr.\@ 786854 G-Statistics from the European Research Council under the European Union's Horizon 2020 research and innovation program. He was also supported by the French government through the 3IA Côte d’Azur Investments ANR-19-P3IA-0002 managed by the National Research Agency (ANR). 
S. Sommer is supported by the Villum Foundation Grants 40582 and the Novo Nordisk Foundation grant NNF18OC0052000.

\bibliographystyle{abbrv}
\bibliography{configuration-spaces}

\begin{thebibliography}{10}

\bibitem{afsari_riemannian_2011}
B.~Afsari.
\newblock Riemannian $l^p$ center of mass: {Existence}, uniqueness, and
  convexity.
\newblock {\em Proceedings of the American Mathematical Society},
  139(02):655--673, 2 2011.

\bibitem{Michor03orbit}
D.~Alekseevsky, A.~Kriegl, M.~Losik, and P.~W. Michor.
\newblock The {R}iemannian geometry of orbit spaces---the metric, geodesics,
  and integrable systems.
\newblock {\em Publ. Math. Debrecen}, 62(3-4):247--276, 2003.
\newblock Dedicated to Professor Lajos Tam{\'a}ssy on the occasion of his 80th
  birthday.

\bibitem{AKPetrunin2019}
S.~Alexander, V.~Kapovitch, and A.~Petrunin.
\newblock {\em An invitation to {A}lexandrov geometry: CAT(0) spaces}.
\newblock Springer Briefs in Mathematics. Springer, Cham, 2019.

\bibitem{arnaudon_means_2014}
M.~Arnaudon and L.~Miclo.
\newblock Means in complete manifolds: uniqueness and approximation.
\newblock {\em ESAIM: Probability and Statistics}, 18:185--206, 2014.

\bibitem{bhattacharya_nonparametric_2002}
R.~Bhattacharya and V.~Patrangenaru.
\newblock Nonparametric estimation of location and dispersion on {Riemannian}
  manifolds.
\newblock {\em Journal of Statistical Planning and Inference}, 108(1-2):23--35,
  11 2002.

\bibitem{bhattacharya_large_2003}
R.~Bhattacharya and V.~Patrangenaru.
\newblock Large sample theory of intrinsic and extrinsic sample means on
  manifolds.
\newblock {\em The Annals of Statistics}, 31(1):1--29, 2 2003.

\bibitem{bhattacharya_large_2005}
R.~Bhattacharya and V.~Patrangenaru.
\newblock Large sample theory of intrinsic and extrinsic sample means on
  manifolds\textemdash{{II}}.
\newblock {\em The Annals of Statistics}, 33(3):1225--1259, 6 2005.

\bibitem{BHV01}
L.~J. Billera, S.~P. Holmes, and K.~Vogtmann.
\newblock Geometry of the space of phylogenetic trees.
\newblock {\em Adv. in Appl. Math.}, 27(4):733--767, 2001.

\bibitem{blum_central_1958}
J.~R. Blum, H.~Chernoff, M.~Rosenblatt, and H.~Teicher.
\newblock Central {{Limit Theorems}} for {{Interchangeable Processes}}.
\newblock {\em Canadian Journal of Mathematics}, 10:222--229, 1958/ed.

\bibitem{burago2001course}
D.~Burago, Y.~Burago, and S.~Ivanov.
\newblock {\em A course in metric geometry}, volume~33 of {\em Graduate Studies
  in Mathematics}.
\newblock American Mathematical Society, Providence, RI, 2001.

\bibitem{buser_gromovs_1981}
P.~Buser and H.~Karcher.
\newblock {\em Gromov's almost flat manifolds}.
\newblock Number~81 in Ast\'erisque. Soci\'et\'e math\'ematique de France,
  1981.

\bibitem{cardaliaguet2013notes}
P.~Cardaliaguet.
\newblock Notes on mean field games (from {P.-L. Lions'} lectures at coll\`ege
  de france).
\newblock Available on the website of Coll\`ege de France
  \url{http://www.college-de-france.fr}.

\bibitem{chernoff_central_1958}
H.~Chernoff and H.~Teicher.
\newblock A {{Central Limit Theorem}} for {{Sums}} of {{Interchangeable Random
  Variables}}.
\newblock {\em Annals of Mathematical Statistics}, 29(1):118--130, 3 1958.

\bibitem{chow_probability_1997}
Y.~S. Chow and H.~Teicher.
\newblock {\em Probability {{Theory}}: {{Independence}},
  {{Interchangeability}}, {{Martingales}}}.
\newblock Springer {{Texts}} in {{Statistics}}. {Springer}, {New York}, third
  edition, 1997.

\bibitem{daprato2014stochastic}
G.~Da~Prato and J.~Zabczyk.
\newblock {\em Stochastic equations in infinite dimension}.
\newblock Cambridge University Press, 2 edition, 2014.

\bibitem{finetti_prevision_1937}
B.~de~Finetti.
\newblock La {{Pr\'evision}}: {{Ses Lois Logiques}}, {{Ses Sources
  Subjectives}}.
\newblock {\em Annales de l'Institut Henri Poincar\'e}, 17:1--68, 1937.

\bibitem{dellacherie1975ensembles}
C.~Dellacherie.
\newblock Ensembles analytiques: th{\'e}or{\`e}mes de s{\'e}paration et
  applications.
\newblock In {\em S{\'e}minaire de Probabilit{\'e}s IX Universit{\'e} de
  Strasbourg}, volume 465 of {\em Lecture Notes in Mathematics}, pages
  336--372. Springer, 1975.

\bibitem{diaconis_finite_1980}
P.~Diaconis and D.~Freedman.
\newblock Finite {Exchangeable} {Sequences}.
\newblock {\em The Annals of Probability}, 8(4):745--764, 1980.
\newblock Publisher: Institute of Mathematical Statistics.

\bibitem{dryden_statistical_1998}
I.~Dryden and K.~Mardia.
\newblock {\em Statistical {{Shape Analysis}}}.
\newblock {John Wiley \& Sons}, 1998.

\bibitem{eltzner_geometrical_2019}
B.~Eltzner.
\newblock Geometrical {Smeariness} - {A} new {Phenomenon} of {Fr\'echet}
  {Means}.
\newblock {\em arXiv:1908.04233}, 8 2019.

\bibitem{eltzner_smeary_2019}
B.~Eltzner and S.~F. Huckemann.
\newblock A smeary central limit theorem for manifolds with application to
  high-dimensional spheres.
\newblock {\em Annals of Statistics}, 47(6):3360--3381, 12 2019.

\bibitem{figalli_convexity_2015}
A.~Figalli, T.~O. Gallou\"et, and L.~Rifford.
\newblock On the {Convexity} of {Injectivity} {Domains} on {Nonfocal}
  {Manifolds}.
\newblock {\em SIAM Journal on Mathematical Analysis}, 47(2):969--1000, 1 2015.

\bibitem{finetti_funzione_1930}
B.~D. Finetti.
\newblock {\em {Funzione caratteristica di un fenomeno aleatorio}}.
\newblock {Societ\`a anonima tipografica}, 1930.

\bibitem{fournier2015rate}
N.~Fournier and A.~Guillin.
\newblock On the rate of convergence in wasserstein distance of the empirical
  measure.
\newblock {\em Probability Theory and Related Fields}, 162(3-4):707--738, 2015.

\bibitem{frechet_les_1948}
M.~Fr\'echet.
\newblock Les \'el\'ements al\'eatoires de nature quelconque dans un espace
  distancie.
\newblock {\em Ann. Inst. H. Poincar\'e}, 10:215--310, 1948.

\bibitem{gromov1999metric}
M.~Gromov.
\newblock {\em Metric structures for Riemannian and non-Riemannian spaces,
  Based on the 1981 French original, With appendices by M. Katz, P. Pansu and
  S. Semmes. Translated from the French by Sean Michael Bates}, volume 152 of
  {\em Progress in Mathematics}.
\newblock Birkh\"auser, Boston, 1999.

\bibitem{hewitt_symmetric_1955}
E.~Hewitt and L.~J. Savage.
\newblock Symmetric measures on {{Cartesian}} products.
\newblock {\em Transactions of the American Mathematical Society},
  80(2):470--501, 1955.

\bibitem{hotz_sticky_2013}
T.~Hotz, S.~Huckemann, H.~Le, J.~S. Marron, J.~C. Mattingly, E.~Miller,
  J.~Nolen, M.~Owen, V.~Patrangenaru, and S.~Skwerer.
\newblock Sticky central limit theorems on open books.
\newblock {\em The Annals of Applied Probability}, 23(6):2238--2258, 12 2013.

\bibitem{huckemann_intrinsic_2011}
S.~F. Huckemann.
\newblock Intrinsic inference on the mean geodesic of planar shapes and tree
  discrimination by leaf growth.
\newblock {\em Ann. Statist}, page 2011, 2011.

\bibitem{huckemann_meaning_2012}
S.~F. Huckemann.
\newblock On the meaning of mean shape: manifold stability, locus and the two
  sample test.
\newblock {\em Annals of the Institute of Statistical Mathematics},
  64(6):1227--1259, 12 2012.

\bibitem{hundrieser_finite_2020}
S.~Hundrieser, B.~Eltzner, and S.~F. Huckemann.
\newblock Finite {Sample} {Smeariness} of {Fr{\'e}chet} {Means} and
  {Application} to {Climate}.
\newblock {\em arXiv:2005.02321}, 5 2020.

\bibitem{jain_data_2010}
A.~K. Jain.
\newblock Data clustering: 50 years beyond {K}-means.
\newblock {\em Pattern Recognition Letters}, 31(8):651--666, 6 2010.

\bibitem{jupp89}
P.~E. Jupp and K.~V. Mardia.
\newblock A unified view of the theory of directional statistics, 1975-1988.
\newblock {\em International Statistical Review}, 57(3):261--294, 1989.

\bibitem{karcher_riemannian_1977}
H.~Karcher.
\newblock Riemannian center of mass and mollifier smoothing.
\newblock {\em Communications on Pure and Applied Mathematics}, 30(5):509--541,
  1977.

\bibitem{kendall_shape_1984}
D.~G. Kendall.
\newblock Shape {{Manifolds}}, {{Procrustean Metrics}}, and {{Complex
  Projective Spaces}}.
\newblock {\em Bull. London Math. Soc.}, 16(2):81--121, 3 1984.

\bibitem{kendall_shape_1999}
D.~G. Kendall, D.~Barden, T.~K. Carne, and H.~Le, editors.
\newblock {\em Shape \& {Shape} {Theory}}.
\newblock Wiley {Series} in {Probability} and {Statistics}. John Wiley \& Sons,
  Inc., Hoboken, NJ, USA, 10 1999.

\bibitem{kendall_limit_2011}
W.~S. Kendall and H.~Le.
\newblock Limit theorems for empirical {{Fr\'echet}} means of independent and
  non-identically distributed manifold-valued random variables.
\newblock {\em Brazilian Journal of Probability and Statistics},
  25(3):323--352, 11 2011.

\bibitem{klass_central_1987}
M.~Klass and H.~Teicher.
\newblock The {{Central Limit Theorem}} for {{Exchangeable Random Variables
  Without Moments}}.
\newblock {\em Annals of Probability}, 15(1):138--153, 1 1987.

\bibitem{kuwae2008variational}
K.~Kuwae and T.~Shioya.
\newblock Variational convergence over metric spaces.
\newblock {\em Transactions of the American Mathematical Society},
  360(1):35--75, 2008.

\bibitem{macqueen_methods_1967}
J.~MacQueen.
\newblock Some methods for classification and analysis of multivariate
  observations.
\newblock In {\em Proceedings of the fifth Berkeley symposium on mathematical
  statistics and probability}, volume~1, pages 281--297. University of
  California Press, 1967.
\newblock ISSN: 0097-0433.

\bibitem{mantegazza2002hamilton}
{Mantegazza} and {Mennucci}.
\newblock Hamilton\textemdash{{Jacobi Equations}} and {{Distance Functions}} on
  {{Riemannian Manifolds}}.
\newblock {\em Applied Mathematics \& Optimization}, 47(1):1--25, 12 2002.

\bibitem{marron_overview_2014}
J.~S. Marron and A.~M. Alonso.
\newblock Overview of object oriented data analysis.
\newblock {\em Biometrical Journal}, 56(5):732--753, 9 2014.

\bibitem{Michor08}
P.~W. Michor.
\newblock {\em Topics in differential geometry}, volume~93 of {\em Graduate
  Studies in Mathematics}.
\newblock American Mathematical Society, Providence, RI, 2008.

\bibitem{panaretos2020invitation}
V.~M. Panaretos and Y.~Zemel.
\newblock {\em An Invitation to Statistics in Wasserstein Space}.
\newblock Springer Nature, 2020.

\bibitem{pennec_intrinsic_2006}
X.~Pennec.
\newblock Intrinsic {{Statistics}} on {{Riemannian Manifolds}}: {{Basic Tools}}
  for {{Geometric Measurements}}.
\newblock {\em J. Math. Imaging Vis.}, 25(1):127--154, 2006.

\bibitem{pennec_curvature_2019}
X.~Pennec.
\newblock Curvature effects on the empirical mean in {{Riemannian}} and affine
  {{Manifolds}}: A non-asymptotic high concentration expansion in the
  small-sample regime.
\newblock {\em arXiv:1906.07418}, 6 2019.

\bibitem{pennec_sommer_fletcher_2020}
X.~Pennec, S.~Sommer, and P.~T. Fletcher.
\newblock {\em Riemannian {Geometric} {Statistics} in {Medical} {Image}
  {Analysis}}.
\newblock Elsevier, 2020.

\bibitem{sturm_probability_2003}
K.-T. Sturm.
\newblock Probability measures on metric spaces of nonpositive curvature.
\newblock In P.~Auscher, T.~Coulhon, and A.~Grigor'yan, editors, {\em Heat
  {Kernels} and {Analysis} on {Manifolds}, {Graphs}, and {Metric} {Spaces},
  {Paris}, {France}}, volume 338, pages 357--390. American Mathematical
  Society, Providence, Rhode Island, 2003.

\bibitem{turaga_riemannian_2016}
P.~K. Turaga and A.~Srivastava, editors.
\newblock {\em Riemannian {Computing} in {Computer} {Vision}}.
\newblock Springer International Publishing, Cham, 2016.

\bibitem{vandervaart2000asymptotic}
A.~W. Van~der Vaart.
\newblock {\em Asymptotic statistics}, volume~3.
\newblock Cambridge university press, 2000.

\bibitem{villani2008optimal}
C.~Villani.
\newblock {\em Optimal transport: old and new}.
\newblock Springer, 2008.

\bibitem{wiki:orbifolds}
{Wikipedia contributors}.
\newblock Orbifolds --- {W}ikipedia{,} the free encyclopedia, 2020.
\newblock [Online; accessed 13-July-2020].

\bibitem{willard2004general}
S.~Willard.
\newblock {\em General topology}.
\newblock Courier Corporation, 2004.

\bibitem{xu_power_2019}
J.~Xu and K.~Lange.
\newblock Power k-{Means} {Clustering}.
\newblock In K.~Chaudhuri and R.~Salakhutdinov, editors, {\em International
  Conference on Machine Learning}, volume~97 of {\em Proceedings of {Machine}
  {Learning} {Research}}, pages 6921--6931, Long Beach, California, USA, 6
  2019. PMLR.

\bibitem{ziezold_expected_1977}
H.~Ziezold.
\newblock On {{Expected Figures}} and a {{Strong Law}} of {{Large Numbers}} for
  {{Random Elements}} in {{Quasi}}-{{Metric Spaces}}.
\newblock In J.~Ko{\v z}e{\v s}nik, editor, {\em Transactions of the {{Seventh
  Prague Conference}} on {{Information Theory}}, {{Statistical Decision
  Functions}}, {{Random Processes}} and of the 1974 {{European Meeting}} of
  {{Statisticians}}: Held at {{Prague}}, from {{August}} 18 to 23, 1974}, pages
  591--602. {Springer Netherlands}, {Dordrecht}, 1977.

\end{thebibliography}

\end{document}